\documentclass[11pt]{amsart}
\usepackage{bbm}
\usepackage{amsmath, latexsym, amssymb}
\usepackage{ulem} 
\usepackage[usenames]{color} 
\usepackage[all, knot]{xy}

\newcommand\hbr{\sigma}

\def\1{\mathbf{1}}
\def\a{\alpha}
\def\b{\beta}
\def\e{\epsilon}
\def\be{{\boldsymbol\epsilon}}

\def\sgns{\operatorname{sgn}(s)}
\def\w{\theta}
\def\G{\Gamma}
\def\g{\gamma}

\def\s{\sigma}
\def\hs{\hat\sigma}
\def\es{\epsilon_\sigma}
\def\bhs{\hat{\boldsymbol\sigma}}
\def\orho{\ol \rho}
\def\l{\lambda}

\def\k{\Bbbk}
\def\bm1{\mathbbm{1}}

\def\fs{\mathfrak{s}}
\def\ft{\mathfrak{t}}
\def\fg{\gamma}
\def\ff{\mathfrak{f}}

\def\fd{\mathfrak{d}}

\def\BQ{{\mathbb{Q}}}
\def\BQA{{\mathbb{Q}_{\operatorname{ab}}}}
\def\qdim{{\operatorname{qdim}}}
\def\BC{{\mathbb{C}}}

\def\BK{{\mathbb{K}}}
\def\BZ{{\mathbb{Z}}}
\def\BR{{\mathbb{R}}}

\def\BN{{\mathbb{N}}}
\def\BS{{\tilde{\mathbf{s}}}}
\def\bs{{\mathbf{s}}}
\def\BT{{\tilde{\mathbf{t}}}}
\def\bX{{\mathbf{X}}}

\def\bc{{\mbox{\bf c}}}

\def\qed{{\ \ \ \mbox{$\square$}}}
\newcommand\mtx[1]{\begin{bmatrix} #1 \end{bmatrix}}
\newcommand{\pmtx}[1]{\begin{pmatrix}#1\end{pmatrix}}
\newcommand\Itw[3]{\pmtx{ #1 \\ #2\,\, #3}}

\def\wt{{\mbox{\rm wt}}}
\def \h{\mathfrak{h}}

\def\tr{{\mbox{\rm tr}}}
\def\Hom{{\mbox{\rm Hom}}}
\def\End{{\mbox{\rm End}}}
\def\Aut{{\mbox{\rm Aut}}}
\def\PGL{{\mbox{\rm PGL}}}
\def\PSL{{PSL_2(\BZ)}}
\def\Gal{{\mbox{\rm Gal}}}
\def\FSexp{{\mbox{\rm FSexp}}}

\def\gs{G_\sigma}

\def\ts{\tilde{s}}
\def\tt{\tilde{t}}

\newcommand\Id{\operatorname{Id}}
\newcommand\FPdim{\operatorname{FPdim}}

\newcommand\Rep{\operatorname{Rep}}

\newtheorem{thm}{Theorem}[section]
\newtheorem{cor}[thm]{Corollary}
\newtheorem{prop}[thm]{Proposition}
\newtheorem{lem}[thm]{Lemma}

\newtheorem{defn}[thm]{Definition}
\newtheorem{example}[thm]{Example}
\newtheorem{remark}[thm]{Remark}

\newtheorem{q}[thm]{Question}
\numberwithin{equation}{section}

\newcommand\ord{\operatorname{ord}}
\newcommand\im{\operatorname{im}}

\newcommand\Z{\BZ}
\newcommand\ol[1]{\overline{#1}}
\newcommand\replace[1]{}
\newcommand\ptr{\operatorname{\underline{ptr}}}

\newcommand\ptrl{\ptr^\ell}
\newcommand\ptrr{\ptr^r}

\newcommand\id{\operatorname{id}}
\renewcommand\o{\otimes}

\newcommand\Tr{\operatorname{Tr}}
\newcommand{\SL}[1]{SL_2(\BZ_{#1})}
\newcommand{\GLk}[1]{GL_{#1}( \k)}
\newcommand{\GLC}[1]{GL_{#1}(\BC)}
\newcommand{\GLR}[2]{GL_{#1}(#2)}
\newcommand{\PGLk}[1]{PGL_{#1}(\k)}

\newcommand\inv{^{-1}}
\DeclareMathOperator\ev{{\operatorname{ev}}}
\DeclareMathOperator\db{\operatorname{db}}
\newcommand\CC{\mathcal C}

\newcommand\KK{\mathcal K}

\newcommand\OO{\mathcal O}
\newcommand\A{\mathcal A}

\newcommand\du{^{\vee}}
\newcommand\bidu{^{\vee\vee}}

\newcommand\op{{\operatorname{op}}}

\newcommand\GalQ[1]{\Gal(\BQ_{#1}/\BQ)}
\newcommand\gq[1]{\frac{p_#1^+}{p_#1^-}}
\newcommand\Jac[2]{{#1 \overwithdelims () #2}}
\newcommand\AQab{{\Aut(\BQA)}}

\newcounter{app}

\newtheorem{Applem}[app]{Lemma}
\newtheorem*{CPM}{Theorem I}
\newtheorem*{CPMII}{Theorem II}
\newtheorem*{Ak}{Acknowledgment}

\makeatletter
\def\namelabel#1#2{\@bsphack
  \protected@write\@auxout{}%
         {\string\newlabel{#1.nme}{{#2}{#2}}}%
  \@esphack}

\makeatother
\title[Congruence property in conformal field theory]{Congruence property in conformal field theory}
\author{Chongying Dong}
\address{Department of Mathematics, University of California, Santa Cruz 98064, USA}
 \email{dong@ucsc.edu}
 \thanks{The first and the third authors were partially supported by NSF grants}
\author{Xingjun Lin}
\address{Department of Mathematics, Sichuan University,
 Chengdu China}
 \email{xingjunlin88@gmail.com}
\author{Siu-Hung Ng}
\address{Department of Mathematics, Louisiana State University, Baton Rouge, LA 70803, USA.}
 \email{rng@math.lsu.edu}

\begin{document}
\maketitle
\begin{abstract}
 The congruence subgroup property is established for the modular representations associated to any modular
 tensor category. This result is used to prove that the kernel of the representation of the modular group on
  the conformal blocks of any rational, $C_2$-cofinite vertex operator algebra is a congruence subgroup.
  In particular,  the $q$-character of each irreducible module is a modular function on the same congruence
  subgroup.  The Galois symmetry of the modular representations is obtained and the order of the anomaly for
  those modular categories satisfying some integrality conditions is determined.
\end{abstract}
\section*{Introduction}

Modular invariance of characters of a rational conformal field theory (RCFT) has been known since the
work of Cardy \cite{Cardy}, and it was proved by Zhu \cite{Z} for rational and $C_2$-cofinite vertex operator
algebras (VOA), which constitute a mathematical formalization of RCFT. The associated matrix representation of
$\SL{}$
relative to the distinguished basis, formed by the trace functions of the irreducible modules or primary
fields, is a powerful tool in the study of vertex operator algebras and conformal field theory.
This matrix representation conceives many intriguing arithmetic properties, and the Verlinde formula is
certainly a notable example \cite{Ver88}. Moreover, it has been shown that these matrices representing
the modular
group
are defined over a certain cyclotomic field \cite{BG91}.

An important characteristic of the modular representation $\rho$ associated with a RCFT is its kernel.
It
has been conjectured by many authors that the kernel is a congruence subgroup of a certain level $n$
(cf.
\cite{Moore87, E95, ES95, DM96, BCIR}). Eholzer further conjectured that this representation is defined
over the $n$-th cyclotomic field $\BQ_n$. In this case, the Galois group $\Gal(\BQ_n/\BQ)$ acts on the
representation
$\rho$ by its entry-wise action. Coste and Gannon proved that $\rho$ determines a signed
permutation matrix $\gs$ for each automorphism $\s$ of $\BQ_n$ \cite{CG}. They also conjectured that the
representation $\s^2 \rho$ is equivalent to $\rho$ under the intertwining operator $\gs$. These
conjectural
properties were summarized as the congruence property of the modular data associated with RCFT
in \cite{CG99, GanBook}. These remarkable properties of RCFT were established by Bantay \cite{Bantay03} under
certain assumptions, and by Coste and Gannon \cite{CG} under the condition that the order of the
Dehn-twist
is odd. In the formalization of RCFT through conformal nets, the congruence property was proved by Xu
\cite{X2}.

In this paper we give a positive answer to the conjecture on the congruence  property for
a rational and $C_2$-cofinite vertex operator algebra $V.$  Such a $V$ has only finitely many irreducible
modules \cite{DLM2}
$M^0,...,M^p$ up to isomorphism and there exist $\lambda_i \in \mathbb{C}$ for $i=0,...,p$ such that
$$M^i=\bigoplus_{n=0}^{\infty}M^i_{\lambda_i +n}$$
 where $M^i_{\lambda_i}\neq 0$ and $L(0)|_{M^i_{\lambda_i+n}}=\lambda_i+n$ for any $n\in\Z.$ Moreover,
 $\lambda_i$ and the central charge $c$ are rational numbers (see \cite{DLM4}).

The trace function for $v\in V_k$ on $M^i$ is defined as
$$Z_i(v,q)=q^{\lambda_i-c/24}\sum_{n=0}^{\infty}(\tr_{M^i_{\lambda_i+n}}o(v))q^{n}$$
where $o(v)=v_{k-1}$ is the $k-1$-st component operator of $Y(v,z)=\sum_{n\in \Z}v_nz^{-n-1}$ which maps each
homogeneous subspace of $M^i$ to itself. If $v=\bm1$ is the vacuum vector we get the $q$-character $\chi_i(q)$
of $M^i.$ It is proved in \cite{Z} that if $V$ is $C_2$-cofinite then $Z_i(v,q)$ converges to a holomorphic
function on the upper half plane in variable $\tau$ where $q=e^{2\pi i\tau}.$  By abusing the notation we also
denote this holomorphic function by $Z_i(v,\tau).$ There is another vertex operator algebra structure
on $V$ \cite{Z} with grading $V=\oplus_{n\in\Z}V_{[n]}.$ We will write $\wt[v]=n$ if $v\in V_{[n]}.$
Then there is a representation $\rho_V$ of the modular group $SL_2(\Z)$ on the space spanned by
$\{Z_i(v,\tau)|i=0...,p\}:$
$$Z_i(v,\gamma\tau)=(c\tau+d)^{\wt[v]}\sum_{j=0}^p\gamma_{ij}Z_j(v,\tau)$$
where $\gamma=\mtx{a & b\\ c& d}\in SL_2(\Z)$ and
$\rho_V(\gamma)= [\gamma_{ij}]$ \cite{Z}.

Here is the first main theorem in this paper.
\begin{CPM}\label{t:II} Let $V$ be a rational, $C_2$-cofinite, self dual simple vertex operator algebra. Then
each $Z_i(v,\tau)$ is a modular form of weight $\wt[v]$ on a congruence subgroup of $SL_2(\BZ)$ of level $n$
which is the smallest positive integer such that $n(\lambda_i-c/24)$ is an integer for all $i.$ In particular,
each $q$-character $\chi_i$ is a modular function on the same congruence subgroup.
\end{CPM}

We should remark that the modularity of the $q$-characters of irreducible modules for some known vertex
operator algebras such as those associated to the highest weight unitary representations for Kac-Moody
algebras
\cite{KP}, \cite{K} and the Virasoro algebra \cite{Ro} were previously known. The readers are referred to
\cite{DMN} for the modularity of $Z_i(v,\tau)$ when $V$ is a vertex operator algebra associated to a positive
definite even lattice.

According to \cite{H2, H3}, the category $\CC_V$ of modules of a rational and $C_2$-cofinite vertex operator
algebra $V$ under the tensor product defined in \cite{HL1, HL2, HL3, H1} is a modular tensor category over
$\BC$. To establish this theorem we have to turn our attention to general modular tensor categories.

Modular tensor categories, or simply called modular categories, play an integral role in the
Reshetikhin-Turaev TQFT invariant of 3-manifolds \cite{Turaev}, and topological quantum computation \cite{Wa}.
They also constitute another
formalization
of RCFT \cite{MS90, BaKi}.

Parallel to a rational conformal field theory, associated to a modular category
$\A$ are the invertible matrices $\ts$ and $\tt$
indexed
by the set $\Pi$ of isomorphism classes of simple objects of $\A$. These matrices define a projective
representation $\orho_\A$ of $\SL{}$ by the assignment
$$
\fs:=\mtx{0 & -1\\ 1 & 0} \mapsto \ts \quad\text{and}\quad \ft:=\mtx{1 & 1\\ 0 & 1} \mapsto \tt,
$$
and the well-known presentation $\SL{} = \langle \fs, \ft \mid \fs^4=1, (\fs \ft)^3 =\fs^2\rangle$ of
the
modular group. It was proved by Ng and Schauenburg in \cite{NS4} that the kernel of this projective
representation of $\SL{}$ is
a congruence subgroup of level $N$ where $N$ is the order of $\tt$. Moreover, both $\ts$ and $\tt$ are
matrices
over
$\BQ_N$. For  factorizable semisimple Hopf algebras, the corresponding result was proved previously by
Sommerh\"auser
and
Zhu \cite{SZh}.

The projective representation  $\orho_\A$ can be lifted to an ordinary representation
 of $\SL{}$ which is called a \emph{modular representation of} $\A$ in \cite{NS4}. There are only
finitely
many modular representations of $\A$ but,  in general, none of them is a canonical choice. However, if
$\A$
is the Drinfeld center of a spherical fusion category, then $\A$ is
modular (cf. \cite{MugerII03}) and it admits a canonical modular representation defined over $\BQ_N$
whose kernel is a congruence
subgroup of level $N$ (cf. \cite{NS4}). The canonical modular representation of the module
category over the Drinfeld double of a semisimple Hopf algebra was shown to have a
congruence
kernel as well as Galois symmetry (see Theorem II (iii) and (iv)) in \cite{SZh}.

The second main theorem of this paper is to prove that
the congruence property and Galois symmetry holds for all modular representations of any modular
category.

\begin{CPMII}\label{t:I}
  Let $\A$ be a modular  category over any algebraically closed field $\k$ of characteristic zero with the set of
  isomorphism classes of simple objects $\Pi$,
  and Frobenius-Schur exponent $N$. Suppose $\rho: \SL{} \to \GLk{\Pi}$ is a modular representation of
  $\A$ where $\GLk{\Pi}$ denotes the group of invertible matrices over $\k$ indexed by $\Pi$. Set
  $s = \rho(\fs)$ and $t = \rho(\ft)$. Then:
  \begin{enumerate}
    \item [(i)]$\ker \rho$ is a congruence subgroup of level $n$ where $n =\ord(t)$. Moreover, $N
        \mid
        n \mid 12N$.
    \item[(ii)] $\rho$ is $\BQ_n$-rational, i.e. $\im \rho \le GL_\Pi(\BQ_n)$, where $\BQ_n
        =\BQ(\zeta_n)$ for some primitive $n$-th root of unity $\zeta_n \in \k$\,.
    \item[(iii)] For $\sigma \in \Gal(\BQ_n/\BQ)$, $\gs=\sigma(s)s\inv$ is a signed permutation
        matrix,
        and
    $$
      \s^2(\rho(\fg)) = \gs \rho(\fg) \gs\inv
    $$
    for all $\fg \in \SL{}$. In particular, if $(\gs)_{ij}=\es(i)\delta_{\hs(i)j}$ for some sign function $\es$ and  permutation $\hs$ on $\Pi$, then $\s^2(t_{ii}) = t_{\hs(i)\hs(i)}$ for all $i \in \Pi$.
    \item[(iv)] Let $a$ be an integer relatively prime to $n$ with an inverse $b$ modulo $n$.  For
        the
        automorphism $\s_a$ of $\BQ_n$ given by $\zeta_n \mapsto \zeta_n^a$,
         $$
         G_{\sigma_a} = t^a s t^b s t^a s\inv\,.
        $$
   \end{enumerate}
\end{CPMII}

We now return to the modular tensor category $\CC_V$ associated to a rational, $C_2$-cofinite and self dual
vertex operator algebra $V$. This yields a projective representation of $SL_2(\Z)$ on space spanned by the
equivalence classes of irreducible $V$-modules. We show in Theorem \ref{iden}  that the representation $\rho_V$
of $SL_2(\Z)$ is a modular representation  of $\CC_V$. This implies that the kernel of $\rho_V$ is a congruence
subgroup of $SL_2(\Z).$

Although the congruence property proved in Theorem II is motivated by solving the congruence property
conjecture on the trace functions of vertex operator algebras, the result has its own importance.  We will
discuss this in the rest of the introduction.

It was also shown in \cite{SZh} that the (unnormalized) $T$-matrix $\tt$ of the module category over a factorizable Hopf
algebra
also enjoys the Galois symmetry $\s^2(\tt) = \gs \tt \gs\inv$ for any $\s \in \Gal(\BQ_N/\BQ)$. However, this extra
 symmetry does not hold for a general modular  category $\A$ (see Example \ref{ex:fail T sym}). This
condition is, in fact, related to the order of the quotient of the Gauss sums, called the
\emph{anomaly}, of $\A$. It is proved in Proposition \ref{p:anomaly} that Galois symmetry of the
$T$-matrix is equivalent to the condition that the anomaly is a fourth root of unity. We will prove in Proposition
\ref{p:anomaly1} that the anomaly of any \emph{integral} modular category  is always a fourth root of
unity. Therefore, the $T$-matrix of any integral modular category enjoys the Galois symmetry. For a
\emph{weakly integral} modular category, such as the Ising model, the anomaly is always an eighth root of
unity (Theorem \ref{t:weakly integral}).

  Using Theorem II, we  uncover some relations among the global dimension $\dim \A$, the Frobenius-Schur
  exponent $N$ and the order of the anomaly $\a$ of a modular  category $\A$. We define $J_\A =
  (-1)^{1+\ord \a}$ to record the parity of the order of the anomaly. If $N$ is not a multiple of $4$,
  then
  $J_\A \dim \A$ has a square root in $\BQ_N$. If, in addition, $\dim \A$ is an odd integer, then $J_\A$
  coincides with the Jacobi symbol $\Jac{-1}{\dim \A}$. The consequence of this observation is a result
  closely related to the Cauchy theorem of integral fusion category.

  The organization of this paper is as follows: Section \ref{s:prelim} covers some basic definitions,
  conventions and preliminary results on spherical fusion categories and modular categories. In Section
  \ref{s:tI i and ii}, we prove the congruence property, Theorem
  II
  (i) and (ii), by proving a lifting theorem of modular projective representations with congruence
  kernels. In Section \ref{s:voa}, we prove the associated representation of modular invariance of trace
  functions of a rational, $C_2$-cofinite vertex operator algebra $V$ is a modular representation of its module
  category. Using  Theorem II (i) and (ii) obtained in Section \ref{s:tI i and ii}, we prove Theorem I: The
  trace functions of $V$ are modular forms.
  In Section \ref{s:Gal1}, we assume the technical Lemma \ref{l:double case} to prove the Galois
  symmetry
  of modular categories as well as RCFTs, Theorem II (iii) and (iv). Section \ref{s:Quantum_double} is devoted
  to the
  proof
  of Lemma \ref{l:double case} by using generalized Frobenius-Schur indicators. In Section
  \ref{s:anomaly},
  we use the congruence property and Galois symmetry of modular categories (Theorem II) to uncover some
  arithmetic relations among the global dimension, the Frobenius-Schur exponent and the anomaly of a
  modular category. In particular, we determine the order of the anomaly of a modular category
  satisfying
  certain integrality conditions.

\section{Basics of modular tensor categories}\label{s:prelim}
In this section, we will collect some conventions and  preliminary results on spherical fusion
categories and modular categories. Most of these results are
quite well-known, and the readers are referred to  \cite{Turaev, BaKi, NS1, NS2, NS3, NS4} and the
references therein.

Throughout this paper, $\k$ is always assumed to be an algebraically closed field of characteristic
zero,
and the group of invertible matrices over a commutative ring $K$ indexed by $\Pi$ is denoted by $GL_\Pi(K)$,
and we will write $\PGL_\Pi(K)$ for its associated projective linear group. If $\Pi=\{1, \dots,
r\}$
for some positive integer $r$, then $GL_\Pi(K)$ (resp. $PGL_\Pi(K)$) will be denoted by the standard
notation $GL_r(K)$ (resp. $PGL_r(K)$) instead.

For any primitive $n$-th root of unity $\zeta_n \in \k$,  $\BQ_n:=\BQ(\zeta_n)$ is the smallest
subfield of $\k$ containing all the $n$-th roots of unity in $\k$. Recall that $\GalQ{n} \cong U(\BZ_n)$,
the group of units of $\BZ_n$. Let $a$ be an integer relatively prime to $n$. The associated $\s_a \in
\Gal(\BQ_n/\BQ)$  is defined by
$$
\s_a(\zeta_n)=\zeta_n^a\,.
$$

Let $\BQA =\bigcup_{n \in \BN} \BQ_n$, the abelian closure of $\BQ$ in $\k$. Since $\BQ_n$ is Galois over
$\BQ$, $\s(\BQ_n) =\BQ_n$ for
all automorphisms $\s$ of $\BQA$. Moreover, the restriction map $\AQab \xrightarrow{res} \GalQ{n}$ is
surjective for all positive integers $n$. Thus, for any integer $a$ relatively prime to $n$, there exists
$\s \in \AQab$ such that $\s|_{\BQ_n} = \s_a$.

\subsection{Spherical fusion categories} \label{ss:sphcat}
In a left rigid monoidal category $\CC$ with tensor product $\o$ and unit object $\1$, a left dual
$V\du$
of $V \in \CC$ with morphisms $\db_V: \1 \to V \o V\du$ and $\ev_V: V\du \o V \to \1$  is denoted by the
triple
 $(V\du, \db_V, \ev_V)$. The left duality can be extended to a monoidal functor $(-)\du : \CC \to
 \CC^{\op}$, and so $(-)\bidu: \CC \to \CC$ defines a monoidal equivalence. Moreover we can choose $\1\du =
 \1$. A
 \emph{pivotal structure} of $\CC$ is an isomorphism $j:\Id_\CC \to (-)\bidu$ of monoidal functors. One
 can
 respectively define the  left and the right pivotal traces of an endomorphism $f: V \to V$ in $\CC$ as
\begin{gather*}
\ptrl(f)=\bigg(\1\xrightarrow{\db_{V\du}} V\du\o V\bidu\xrightarrow{\id\o j_V\inv} V\du\o
V\xrightarrow{\id \o f} V\du\o V \xrightarrow{\ev_V}\1\bigg) \text{ and } \\
\ptrr(f)=\bigg(\1\xrightarrow{\db_V}V\o
V\du\xrightarrow{f\o \id}V\o V\du \xrightarrow{j_V\o \id} V\bidu \o V\du
\xrightarrow{\ev_{V\du}}\1\bigg)\,.
\end{gather*}
The pivotal structure is called \emph{spherical} if the two pivotal traces coincide for all endomorphisms
$f$ in $\CC$.

A \emph{pivotal} (resp. \emph{spherical}) \emph{category} $(\CC, j)$ is a left rigid monoidal category
$\CC$ equipped with a pivotal (resp. spherical) structure $j$. We will simply denote the pair $(\CC, j)$
by
$\CC$ when there is no ambiguity. The left and the right pivotal dimensions of $V\in\CC$ are defined as
$d_\ell(V)=\ptr^\ell(\id_V)$ and $d_r(V)=\ptr^r(\id_V)$ respectively. In a
spherical category, the pivotal traces and dimensions will be denoted by
$\ptr(f)$ and $d(V)$ (or $\dim V$), respectively.

A \emph{fusion category} $\CC$ over the field $\k$ is an abelian $\k$-linear semisimple (left) rigid
monoidal category with a simple unit object $\1$, finite-dimensional morphism spaces and finitely many
isomorphism classes of simple objects (cf. \cite{ENO}). We will denote by $\Pi_\CC$ the set of
isomorphism
classes of simple objects of $\CC$, and by $0$ the isomorphism class of $\1$, unless stated otherwise. If
$i
\in \Pi_\CC$, we write $i^*$ for the (left) dual of the isomorphism class $i$. Moreover, $i \mapsto i^*$
defines a permutation of order $\le 2$ on $\Pi_\CC$.

In a spherical fusion category $\CC$  over $\k$, $d(V)$ can be identified with a scalar in $\k$ for $V
\in \CC$. We abbreviate $d_i \in \k$ for the pivotal dimension of $i \in \Pi_\CC$. By \cite[Lem.
2.8]{MuI}, $d_i = d_{i^*}$ for all $i \in \Pi_\CC$. The global dimension $\dim \CC$ of $\CC$ is defined
by
$$
\dim \CC = \sum_{i \in \Pi_\CC} d_i^2\,.
$$

A pivotal category $(\CC, j)$ is said to be \emph{strict} if $\CC$ is a strict monoidal category and the
pivotal structure $j$ as well as the canonical isomorphism $(V \o W)\du �\to  W\du \o  V\du$ are
identities. It has been proved in \cite[Thm. 2.2]{NS1} that every pivotal category is \emph{pivotally
equivalent} to a strict pivotal category.

\subsection{Representations of the modular group}\label{ss:SL2}
The modular group $\SL{}$ is the group of $2\times 2$ integral matrices with determinant 1. It is
well-known that the modular group is generated by
\begin{equation}\label{eq:relation}
\fs= \mtx{0 & -1\\ 1 & 0} \quad\text{and}\quad \ft=\mtx{1 & 1\\ 0 & 1} \text{ with defining relations }
(\fs\ft)^3=\fs^2 \text{ and } \fs^4=\id.
\end{equation}
We denote by $\G(n)$  the kernel of the reduction modulo $n$ epimorphism  $\pi_n:\SL{} \to \SL{n}$. A
subgroup $L$ of $\SL{}$ is called a  \emph{congruence subgroup of level $n$} if $n$ is the least
positive
integer for which $\G(n) \le L$.

For any pair of matrices $A, B$ in $\GLk{r}$, $r \in \BN$, satisfying the conditions
$$
A^4 =\id\quad\text{and}\quad(AB)^3=A^2,
$$
one can define a representation $\rho: \SL{} \to \GLk{r}$ such that  $\rho(\fs)=A$ and $\rho(\ft)=B$ via
the presentation \eqref{eq:relation} of $\SL{}$.

Suppose $\orho: \SL{} \to \PGLk{r}$ is a projective representation of $\SL{}$. A \emph{lifting} of
$\orho$
is an \emph{ordinary} representation $\rho: \SL{} \to \GLk{r}$ such that $\eta \circ \rho =\orho$, where
$\eta: \GLk{r} \to \PGLk{r}$ is the natural surjection map. One can always lift $\orho$ to a
representation
$\rho: \SL{} \to \GLk{r}$ as follows: Let $\hat A, \hat B \in \GLk{r}$ such that
$\orho(\fs)=\eta(\hat A)$ and $\orho(\ft)=\eta(\hat B)$. Then
$$
\hat A^4 = \mu_s \id\quad \text{and}\quad (\hat A \hat B)^3 = \mu_t \hat A^2
$$
for some scalars $\mu_s, \mu_t \in \k^\times$. Take $\l, \zeta \in \k$ such that $\l^4 = \mu_s$ and
$\zeta^3 = \frac{\mu_t}{\l}$, and set $A = \frac{1}{\l} \hat A$ and $B = \frac{1}{\zeta}\hat B$. Then we
have
$$
A^4 = \id \quad \text{and}\quad (A B)^3 = A^2\,.
$$
Therefore, the assignment $\rho: \fs \mapsto A, \ft \mapsto B$ defines a lifting of $\orho$.

Let $\rho$ be a lifting of $\orho$. Suppose $x\in \k$ is a $12$-th root of unity. Then the assignment
\begin{equation}\label{eq:parametrize}
  \rho_x: \fs \mapsto \frac{1}{x^3} \rho(\fs), \quad \ft\mapsto x \rho(\ft)
\end{equation}
also defines a lifting of $\orho$.
If $\rho': \SL{} \to \GLk{r}$ is another lifting of $\orho$, then
$$
\rho'(\fs) = a \rho(\fs) \quad \text{and}\quad \rho'(\ft) = b \rho(\ft)
$$
for some $a, b \in \k^\times$. It follows immediately from \eqref{eq:relation} that $a^4=1$ and $(a
b)^3=a^2$. This implies $b^{12}=1$ and $b^{-3} = a$. Therefore, $\rho'  = \rho_{b}$ and so $\orho$ has
at
most 12 liftings.

For any $12$-th root of unity $x \in \k$, the assignment $\chi_x : \fs \mapsto x^{-3}, \ft \mapsto x$
defines a linear character of $\SL{}$. It is straightforward to check that
$\chi_x \o \rho$ is isomorphic to $\rho_x$ as representations of $\SL{}$. Therefore, the lifting of $\orho$
is
\emph{unique} up to a linear character of $\SL{}$.

\subsection{Modular Categories}\label{ss:MTC}
Following \cite{Kassel}, a \emph{twist}  (or \emph{ribbon structure}) of a left rigid braided monoidal
category $\CC$ with a braiding $c$ is an automorphism $\theta$ of the identity functor $\Id_\CC$
satisfying
$$
\theta_{V \o W} = (\theta_V \o \theta_W)\circ c_{W, V}\circ  c_{V, W}, \quad \theta\du_V = \theta_{V\du}
$$
for $V, W \in \CC$. Associated to the braiding $c$ is the Drinfeld isomorphism $u: \Id_\CC \to (-)\bidu$.
When $\CC$ is a braided fusion category over $\k$, there is a one-to-one correspondence between twists
$\theta$ and spherical  structures $j$ of $\CC$ given by $\theta = u\inv j$ (cf. \cite[p38]{NS3} for more
details).

A \emph{modular tensor category} over $\k$ (cf.  \cite{Turaev, BaKi}), also simply called a
modular category, is a braided spherical fusion category $\A$ over $\k$ such that the $S$-matrix  of
$\A$
defined by
$$
\ts_{ij} = \ptr(c_{V_j, V_{i^*}} \circ c_{V_{i^*}, V_j})\,
$$
is non-singular, where $V_j$ denotes an object in the class $j \in \Pi_\A$. In this case, the associated
ribbon structure $\theta$ is of finite order $N$ (cf. \cite{Vafa88, BaKi}). Let $\theta_{V_i} = \w_i
\id_{V_i}$ for some $\w_i \in \k$. Since $\theta_\1 = \id_\1$, $\w_0 = 1$. The $T$-\emph{matrix} $\tilde{t}$ of
$\A$ is defined by $\tt_{ij} = \delta_{ij} \w_j$ for $i,j \in \Pi_\A$.  It is immediate to see that
$\ord(\tt) = N$ is finite, which is called the \emph{Frobenius-Schur exponent} of $\A$ and denoted by $\FSexp(\A)$ (cf.
\cite[Thm. 7.7]{NS3}).

The matrices $\ts$, $\tt$ of a modular category $\A$ satisfy the conditions:
\begin{equation}\label{eq:STrelations}
(\ts \tt)^3 = p_\A^+ \ts^2, \quad \ts^2=p_\A^+ p_\A^- C, \quad C\tt=\tt C,\quad C^2=\id,
\end{equation}
where $p_\A^{\pm} = \sum_{i\in \Pi_\A} d_i^2 \w_i^{\pm 1} $ are called the \emph{Gauss sums}, and
$C=[\delta_{ij^*}]_{i,j \in \Pi_\A}$ is called the \emph{charge conjugation matrix} of $\A$.
The quotient $\frac{p_\A^+}{p_\A^-}$
is a  root of unity (cf. \cite[Thm. 3.1.19]{BaKi} or \cite{Vafa88}), and
\begin{equation}\label{eq:gausssum}
p_\A^+ p_\A^-  =\dim \A \ne 0.
\end{equation}
Moreover, $\ts$ satisfies
\begin{equation}\label{eq:S}
\ts_{ij} = \ts_{ji} \quad \text{and}\quad \ts_{ij^*} = \ts_{i^*j}
\end{equation}
for all $i, j \in \Pi_\A$.

The relations \eqref{eq:STrelations} imply that
\begin{equation}\label{eq:projrep}
  \orho_\A\colon \fs\mapsto \eta(\ts),\quad \ft\mapsto \eta(\tt)
\end{equation}
defines a projective representation of $\SL{}$, where $\eta:\GLk{\Pi_\A}\to \PGLk{\Pi_\A}$ is the
natural
surjection.   By \cite[Thm. 6.8]{NS4}, $\ker \orho_\A$ is a congruence subgroup of level $N$.

It is well known that $\ol\rho_\A$ can be lifted to an ordinary representation (cf. \cite[Remark 3.1.9]{BaKi} or Section \ref{ss:SL2}).
Following \cite{NS4}, a lifting $\rho$ of $\orho_\A$ is called a \emph{modular representation} of $\A$.
By \eqref{eq:gausssum}, for any  6-th root $\zeta\in \k$ of $\frac{p_\A^+}{p_\A^-}$,
$\left(\frac{p_\A^+}{\zeta^3}\right)^2=\dim \A$. It follows from \eqref{eq:STrelations} that  the
assignment
\begin{equation}\label{eq:repC}
\rho^\zeta: \fs \mapsto \frac{\zeta^3}{p_\A^+} \ts, \quad \ft \mapsto \frac{1}{\zeta} \tt
\end{equation}
defines a  modular representation of $\A$.

Thus, if $\rho$ is a modular representation of $\A$, it follows from  Section \ref{ss:SL2} that $\rho
=
\rho^\zeta_x$ for some $12$-th root of unity $x \in \k$. Thus $\rho(\fs)^2=\pm C$. More precisely,
$\rho(\fs)^2=x^6 C$.

A modular category $\A$ is called \emph{anomaly-free} if the quotient $\gq{\A}=1$. The terminology
addresses the
associated anomaly-free TQFT with such a modular category  \cite{Turaev}. In this spirit, we will simply
call
the quotient $\a_\A:=\gq{\A}$ the \emph{anomaly} of $\A$.

If $\A$ is an anomaly-free modular category, then $p_\A^+$ is a \emph{canonical} choice of square root
of
$\dim \A$, and hence a \emph{canonical} modular representation of $\A$ is determined by the assignment
\begin{equation}\label{eq:canrep}
 \rho_\A :\fs \mapsto \frac{1}{p_\A^+} \ts, \quad \ft \mapsto \tt\,.
\end{equation}

For any modular category $\A$ over $\BC$,  $\dim \A >0$ (cf. \cite{ENO}). The \emph{central charge} $\bc$ of
$\A$
is a rational number modulo 8  defined by $\exp\left(\frac{\pi i \bc}{4}\right)=\frac{p_\A^+}{\sqrt{\dim
\A}}$
where $\sqrt{\dim \A}$ denotes the positive square root of $\dim \A$, and so the anomaly $\a$ of $\A$ is
given by
\begin{equation}\label{eq:charge}
 \a= \exp\left(\frac{\pi i \bc}{2}\right)\,.
\end{equation}
We will show in Theorem \ref{iden} that the central charge $\bc$ of the modular category $\CC_V$ is equal to
central charge $c$ of $V$ modulo 4.
\begin{remark}
{\rm
The $S$ and $T$-matrices of a modular category are preserved by equivalence of braided pivotal
categories
over $\k$, and so are the dimensions of simple objects, the global dimension, the Gauss sums as well as
the
anomaly.  By the last paragraph of Section \ref{ss:sphcat}, without loss of generality, we may assume
that the underlying pivotal category of a modular category over $\k$ is \emph{strict}. }
\end{remark}

\subsection{Quantum doubles of spherical fusion categories} \label{ss:center} Let $\CC$ be a strict
monoidal category. The left Drinfeld center $Z(\CC)$ of $\CC$ is a category whose objects are pairs
 $\mathbf{X}=(X, \hbr_X)$, where $X$ is an object of $\CC$, and the \emph{half-braiding} $\hbr_{X}(-):
 X
 \o (-)\;
\rightarrow \; (-) \o X$ is a natural isomorphism satisfying the
properties $\hbr_X(\1)=\id_X$ and
$$
(\id_V\o \hbr_X(W))\circ (\hbr_X(V) \o \id_W)  =  \hbr_X(V \o W)
$$
for all $V, W \in \CC$. It is well-known that $Z(\CC)$ is a braided strict monoidal category (cf.
\cite{Kassel}) with unit object $(\1, \s_\1)$ and
tensor product
$
  (X, \hbr_X)\o (Y, \hbr_Y) := (X \o Y, \hbr_{X \o Y}),
$
where
$$
\hbr_{X \o Y}(V) =  (\hbr_X(V) \o \id_Y) \circ (\id_X \o \hbr_Y(V)),  \quad \hbr_\1(V)=\id_V
$$
for $V \in \CC$.
The  forgetful functor $Z(\CC)\to\CC, \bX=(X, \s_X) \mapsto X$, is a strict monoidal functor.

When $\CC$ is a (strict) spherical fusion category over $\k$, by M\"uger's result \cite{MugerII03}, the
center $Z(\CC)$ is a modular category over $\k$ with the inherited spherical structure from
$\CC$. In addition,
$$
p_{Z(\CC)}^+ = \dim \CC = p_{Z(\CC)}^-\,.
$$
Therefore, $Z(\CC)$ is anomaly-free and it admits a canonical modular representation $\rho_{Z(\CC)}$
described in \eqref{eq:canrep}. In particular,
\begin{equation}\label{eq:can normalization}
\rho_{Z(\CC)}(\ft) = \tt\quad \text{and}\quad \rho_{Z(\CC)}(\fs) = \frac{1}{\dim \CC} \ts,
\end{equation}
which is called the \emph{canonical normalization} of the $S$-matrix of $Z(\CC)$. By \cite[Thm. 6.7 and 7.1
]{NS4},
$\ker \rho_{Z(\CC)}$ is a congruence subgroup of level $N$, and $\im
\rho_{Z(\CC)} \le \GLR{\Pi_{Z(\CC)}}{\BQ_N}$, where $N  = \ord(\tt)$.

\section{Rationality and Kernels of Modular Representations}\label{s:tI i and ii}
In this section, we will prove the congruence property given in (i) and (ii) of Theorem II. Recall that
a projective representation $\orho: G \to \PGLk{r}$ of a group $G$ determines a cohomology class $\kappa_{\orho}
\in H^2(G, \k^\times)$. For any section $\iota:  \PGLk{r} \to \GLk{r}$ of the natural surjection
$\eta:\GLk{r} \to \PGLk{r}$, the function $\g_\iota: G \times G \to \k^\times$ given by
$$
\rho_\iota(ab) = \g_\iota(a,b) \rho_\iota(a)\rho_\iota(b)
$$
determines a 2-cocycle in $\kappa_{\orho}$, where $\rho_\iota = \iota \circ \orho$. The cohomology class
$\kappa_{\orho}$ is trivial if, and only if, $\orho$ can be \emph{lifted} to a linear representation $\rho: G
\to \GLk{r}$, i.e. $\eta \circ \rho =\ol \rho$ (cf. \cite[p72]{Kap}).

Let $\pi: L \to G$ be a group homomorphism. For any 2-cocycle $\g \in Z^2(G, \k^\times)$,
$\g\circ (\pi \times \pi) \in Z^2(L, \k^\times)$. The assignment $\g \mapsto \g\circ (\pi \times
\pi)$ of 2-cocycles induces the group homomorphism  $\pi^*:H^2(G, \k^\times) \to H^2(L,
\k^\times)$, which is called the inflation map along $\pi$. In particular, $\pi^*\kappa_{\orho} \in H^2(L, \k^\times)$ is associated with the
projective representation $\orho \circ \pi: L \to \PGLk{r}$.

The homology group $H_2(G, \BZ)$ is often called the \emph{Schur multiplier} of $G$ \cite{Weibhobk}.
Since $\k^\times$ is a divisible abelian group, $H^2(G, \k^\times)$ is naturally isomorphic to
$\Hom(H_2(G,
\BZ), \k^\times)$ for any group $G$.  This natural isomorphism allows us to summarize the result of
Beyl \cite[Thm. 3.9 and Cor. 3.10]{Beyl} on the Schur multiplier of $\SL{m}$ as the following
theorem. A proof of the statement is provided for the sake of completeness.  The case for odd integers $m$  was originally proved by Mennicke \cite{Menn}.
\begin{thm} \label{t:Beyl}
Let $\k$ be an algebraically closed field of characteristic zero, and $m$ an integer greater than 1.
Then
$H^2(\SL{m}, \k^\times)\cong \BZ_2$ if $4
\mid m$, and is trivial otherwise. Moreover, the image of the inflation map $\pi^*:H^2(\SL{m},
\k^\times)
\to H^2(\SL{2m}, \k^\times)$ along the natural reduction map $\pi: \SL{2m} \to \SL{m}$ is always
trivial.
\end{thm}
\begin{proof}
  The first statement is a direct consequence of \cite[Thm. 3.9]{Beyl}. For the second statement, it suffices to consider the case $m=2^a q$ with $a \ge 2$ and $q$ odd. Then there are split surjections $p: \SL{m} \to \SL{2^a}$ and $p': \SL{2m} \to \SL{2^{a+1}}$ by the Chinese Remainder Theorem such that the diagram of group homomorphisms commutes:
  $$
  \xymatrix{
    \SL{2m} \ar[r]^-{p'}\ar[d]_{\pi} & \SL{2^{a+1}} \ar[d]^-{\pi'}\\
    \SL{m} \ar[r]^-{p} & \SL{2^a} \\
  }
  $$
  where $\pi'$ is the reduction map. Applying the functor $H^2(-, \k^\times)$ to this commutative diagram, we obtain the following commutative diagram of abelian groups:
  $$
  \xymatrix{
    H^2(\SL{2m}, \k^\times)  & \ar[l]_-{(p')^*} H^2(\SL{2^{a+1}}, \k^\times)  \\
     H^2(\SL{m}, \k^\times)\ar[u]^{\pi^*}  & \ar[l]^-{p^*} H^2(\SL{2^{a}}, \k^\times)\ar[u]_-{(\pi')^*} \\
  }\,.
  $$
  Since $p$ and $p'$ are split surjections, both  $p^*$ and $(p')^*$ are injective. Hence, by the first statement, they are isomorphisms. By \cite[Cor. 3.10]{Beyl}, $(\pi')^*$ is trivial, and so is $\pi^*$.
\end{proof}

Theorem \ref{t:Beyl} is essential to the proof of the following  lifting lemma for a projective representations of
$\SL{}$.
 \begin{lem}\label{l:existence}
   Suppose $\orho: \SL{} \to \PGLk{r}$ is a projective representation for some positive integer $r$ such
   that $\ker \orho$ is a congruence subgroup of level $n$. Let $\orho_n: \SL{n} \to \PGLk{r}$ be the
   projective representation which satisfies $\orho=\orho_n \circ \pi_n$, where $\pi_n :\SL{} \to \SL{n}$
   is the
   reduction modulo $n$ map, and $\kappa$ denotes the associated 2nd cohomology class in $H^2(\SL{n},
   \k^\times)$. Then
  \begin{enumerate}
    \item[(i)] the class $\kappa$ is trivial if, and only if, $\orho$ admits a lifting whose kernel
        is a congruence subgroup of level $n$.
    \item[(ii)] If $\kappa$ is not trivial, then $4 \mid n$ and $\orho$ admits a lifting whose kernel
        is a congruence subgroup of level $2n$.
  \end{enumerate}
  In particular, there exists a lifting $\rho$ of $\orho$ such that $\ker \rho$ is a congruence subgroup
  containing $\G(2n)$.
 \end{lem}
\begin{proof}
  (i) If $\kappa$ is trivial, there exists a linear representation $\rho_n: \SL{n} \to \GLk{r}$ such that $\eta
  \circ \rho_n = \orho_n$. Then $\rho: = \rho_n \circ \pi_n$ is
  a lifting of $\orho$ since $\eta \circ \rho = \eta \circ \rho_n \circ \pi_n = \orho_n \circ \pi_n = \orho$.
    In particular, $\ker \rho$ is a congruence
  subgroup of level at most $n$. Obviously, $\ker \rho \le \ker \orho$. Since $\ker \orho$ is of level
  $n$, the level of $\ker\rho$ is at least $n$.   Therefore, $\ker \rho$ is of level $n$.

Conversely, assume  $\rho:\SL{} \to \GLk{r}$ is a representation whose kernel is a congruence subgroup
of
level $n$ and $\orho = \eta \circ\rho$. Then $\rho$ factors through $\SL{n}$ and so there exists a linear
representation $\rho_n: \SL{n} \to \GLk{r}$ such that $\rho=\rho_n \circ \pi_n$. Since
$$
\eta \circ \rho_n \circ \pi_n  = \eta \circ \rho = \orho= \orho_n \circ \pi_n,
$$
$\eta \circ \rho_n = \orho_n$. Therefore, $\rho_n$ is a lifting of $\orho_n$ and hence $\kappa$ is trivial.

(ii) Now, we consider the case when $\kappa$ is not trivial. By Theorem \ref{t:Beyl}, $4 \mid n$ and
$\pi^*(\kappa) \in H^2(\SL{2n}, \k^\times)$ is trivial where $\pi: \SL{2n} \to \SL{n}$ is the natural
surjection (reduction) map. The composition $\orho_n \circ \pi : \SL{2n} \to \PGLk{r}$ defines a
projective
representation of $\SL{2n}$, and its associated class in $H^2(\SL{2n}, \k^\times)$ is $\pi^*(\kappa)$.
Since $\pi^*(\kappa)$ is trivial, $\orho_n \circ \pi$ can be lifted to a linear representation $f: \SL{2n} \to
\GLk{r}$, i.e. $\eta \circ f = \orho_n \circ \pi$. Thus, we have
the following
commutative diagram:

 $$
\xymatrix{
\SL{2n} \ar@/^1pc/[rrd]^-{f} \ar@/_1.4pc/[rdd]_-{\pi} & & \\
&\SL{}\ar[rd]_-{\orho}\ar[d]_-{\pi_{n}} \ar[ul]^-{\pi_{2n}} &
\GLk{r}  \ar[d]_-{\eta} \\
& \SL{n} \ar[r]_-{\orho_n} & \PGLk{r}\,.
}
$$
The commutativity of the upper quadrangle is given by
$$
\eta \circ f \circ \pi_{2n} = \orho_n \circ \pi \circ \pi_{2n} = \orho_n \circ \pi_{n} = \orho\,.
$$
Set $\rho = f \circ \pi_{2n}$. Then $\eta \circ \rho = \orho$  and so $\G(2n) \le \ker
\rho$. Suppose $\G(m) \le \ker \rho$ for some positive integer $m < 2n$ and $m \mid 2n$. Then $\G(m)
\le
\ker \rho \le \ker \orho$. Since $\ker \orho$ is of level $n$, $n\mid m$. Thus, $m=n$, and hence $\ker
\rho$ is a congruence subgroup of level $n$. It follows from (i) that $\kappa$ is trivial, a
contradiction.
Therefore, $\ker \rho$ is of level $2n$.
\end{proof}

 Now we can prove the following lifting theorem for the projective representations of $\SL{}$ with congruence
 kernels.
 \begin{thm}\label{t:congruence lifting}
   Suppose $\orho: \SL{} \to \PGLk{r}$ is a projective representation for some positive integer $r$ such
   that $\ker \orho$ is a congruence subgroup of level $n$. Then the kernel of any lifting of $\orho$ is
   a
   congruence subgroup of level $m$ where $n \mid m \mid 12 n$.
 \end{thm}
\begin{proof}
  By Lemma \ref{l:existence}, $\orho$ admits a lifting $\xi$ such that $\ker \xi$ is congruence subgroup
  containing $\G(2n)$. Let $\rho$ be a lifting of $\orho$. By Section \ref{ss:SL2}, $\rho= \xi_x
  \cong
  \chi_x \o \xi$  for some $12$-th root of unity $x \in \k$. Note that $\SL{}/\SL{}' \cong \BZ_{12}$ and
  $\G(12) \le \SL{}'$; see for example \cite[Lem. 1.13]{Beyl}. Therefore, $\G(12) \le  \ker \chi_x$ and hence
  $$
  \ker(\chi_x \o \xi) \supseteq \SL{}' \cap \G(2n) \supseteq \G(12) \cap \G(2n) \supseteq \G(12n).
  $$
  Therefore, $\rho$ has a congruence kernel containing $\G(12n)$ and so $m \mid 12n$. Since $\G(m) \le
  \ker
  \rho \le \ker \orho$ and $\ker \orho$ is of level $n$, $n \mid m$.
\end{proof}
A consequence of Theorem \ref{t:congruence lifting} is a proof for
the statements (i) and (ii) of Theorem II.
\begin{proof}[{\rm \bf Proof of Theorem II (i) and (ii)}] By \cite[Thm. 6.8]{NS4}, the projective modular
representation $\orho_\A$ of a modular category $\A$ over $\k$ has a congruence kernel of level $N$
where
$N$ is the order of the $T$-matrix of $\A$. It follows immediately from Theorem \ref{t:congruence
lifting}
that every modular representation $\rho$ has a congruence kernel of level $n$ where $N \mid n \mid 12N$.
By Lemma \ref{l:app1},  $\ord(\rho(\ft)) = n$. Now the statement Theorem II (ii) follows directly from
\cite[Thm. 7.1]{NS4}.
\end{proof}

The congruence property, Theorem II (i) and (ii), is essential to the proof of Theorem I and to the Galois symmetry of
modular categories in Sections \ref{s:Gal1} and \ref{s:Quantum_double}.

\begin{defn}{\rm Let $\A$ be a modular category over $\k$ with $\FSexp(\A)=N$.
\begin{enumerate}
  \item[(i)] By  virtue of Theorem II (i), a modular representation $\rho$ of $\A$ is said to be of
      level $n$ if $\ord(\rho(\ft))=n$.
  \item[(ii)] The projective modular representation $\orho_\A$ of $\A$ factors through a projective
      representation $\orho_{\A,N}$ of $\SL{N}$. We denote by $\kappa_\A$ the cohomology class in
      $H^2(\SL{N}, \k^\times)$ associated with $\orho_{\A,N}$.
\end{enumerate}
  }
\end{defn}

By Theorem \ref{t:Beyl}, the order of $\kappa_\A$ is at most 2. If  $4 \nmid \FSexp(\A)$, $\kappa_\A$ is
trivial. However, if $4 \mid \FSexp(\A)$, Lemma \ref{l:existence} provides the following criterion to
decide the order of $\kappa_\A$.
\begin{cor}\label{c1}
  Let $\A$ be a modular category over $\k$. Suppose $N=\FSexp(\A)$ and $\zeta \in \k$ is a 6-th root of
  the
  anomaly of $\A$. Then $\kappa_\A$ is trivial if, and only if, $(x/\zeta)^N=1$ for some 12-th root of
  unity $x \in \k$. In this case, $x^3 p_\A^+/\zeta^3 \in \BQ_N$. In particular, if $4 \nmid N$, then
  there
  exists a 12-th root of unity $x \in \k$ such that
  $$
  (x/\zeta)^N=1 \quad\text{and}\quad x^3 p_\A^+/\zeta^3 \in \BQ_N\,.
  $$
\end{cor}
\begin{proof}
  By \eqref{eq:repC}, $\zeta$ determines the modular representation $\rho^\zeta$ of $\A$ given by
  $$
  \rho^\zeta: \fs \mapsto \frac{\zeta^3}{ p_\A^+} \ts, \quad \ft \mapsto \frac{1}{\zeta} \tt\,.
  $$
  By Lemma \ref{l:existence} (i) and the last two paragraphs of Section \ref{ss:SL2}, $\kappa_\A$ is trivial if,
  and only if, there exists a 12-th root of unity $x \in \k$ such that $\rho^\zeta_x$ is a level $N$
  modular representation of $\A$. By Theorem II (i), this is equivalent to $\id =  (\frac{x}{\zeta} \tt)^N$
  or
  $(\frac{x}{\zeta})^N = 1$. In this case, Theorem II (ii) implies $\frac{\zeta^3}{x^3 p_\A^+} \ts \in
  GL_{\Pi_\A}(\BQ_N)$ and
  hence $\frac{\zeta^3}{x^3 p_\A^+}  \in \BQ_N$. The last statement follows immediately from Theorem
  \ref{t:Beyl}.
\end{proof}
The corollary implies some arithmetic relations among the Frobenius-Schur exponent,  the global
dimension
and the anomaly of a modular category. These arithmetic consequences will be discussed in Section
\ref{s:anomaly}.

\section{Modularity of trace functions for rational vertex operator algebras}\label{s:voa}
In this section we prove that the trace functions of a rational, $C_2$-cofinite vertex operator algebra $V$
are modular forms on some congruence subgroup by showing that the representation $\rho_V$ of $\SL{}$, defined by
modular transformation of the trace functions of $V$, is a modular representation of $\CC_V$. The congruence
subgroup property obtained in Section  \ref{s:tI i and ii} is then applied to $\rho_V$ to conclude the
modularity of the trace functions of $V$.

\subsection{Preliminaries}\label{ss:3.1}
In this subsection we briefly review some basics of vertex operator algebras following \cite{FLM}, \cite{FHL},
\cite{DLM1}, \cite{DLM2}, \cite{LL} and \cite{Z}.

Let $V=(V,Y,\bm1,\omega)$ be a vertex operator algebra. Then $V$ is $C_2$-cofinite if the subspace
$C_{2}(V)$ of $V$ spanned by all elements of type $a_{-2}b$ for $a,b $ in $V$ has finite codimension in $V$.
Recall from \cite{DLM2} that $V$ is rational if any admissible module is completely reducible. The component operator $L(n)$ of $Y(\omega,z)=\sum_{n\in\Z}L(n)z^{-n-2}$ will be used frequently.
It is proved in \cite{DLM2} that if $V$ is rational then $V$ has only finitely many irreducible admissible
modules
$M^0,...,M^p$ up to isomorphism and there exist $\lambda_i \in \mathbb{C}$ for $i=0,...,p$ such that
$$M^i=\oplus_{n=0}^{\infty}M^i_{\lambda_i +n}$$
 where $M^i_{\lambda_i}\neq 0$ and $L(0)|_{M^i_{\lambda_i+n}}=\lambda_i+n$ for any $n\in\Z.$ Moreover, if $V$
 is also assumed to be $C_2$-cofinite, then $\lambda_i$ and the central charge $c$ of $V$ are rational numbers
 (see \cite{DLM4}). In this paper we always assume that $V$ is simple and we take $M^0$ to be $V.$

Another important concept is the contragredient module. Let $M = \bigoplus_{\lambda\in
\mathbb{C}}{M_{\lambda}}$ be a $V$-module. Let $M'
= \bigoplus_{\lambda \in \mathbb{C}}{M_\lambda^*}$ be the restricted
dual of $M$. It is proved in [FHL] that $M'=(M', Y')$ is naturally a
$V$-module such that $$\langle Y'(a, z)u', v\rangle  = \langle u', Y(e^{zL(1)}(-z^{-2})^{L(0)}a,
z^{-1})v\rangle ,$$\\for $a\in V, u'\in
M'$ and $v\in M,$ and that $(M')'\simeq M$. Moreover, if $M$ is irreducible, so
is $M'.$ A $V$-module $M$ is said to be self dual if $M$ and $M'$
are isomorphic. In this paper, we'll always assume that the vertex operator algebra $V$ satisfies
the following assumptions:
\begin{enumerate}
\item[(V1)] $V=\bigoplus\limits_{n\geq 0}V_n$ with $\dim V_0=1$ is simple and self dual,\\
\item[(V2)] $V$ is $C_{2}$-cofinite and rational.
\end{enumerate}

The assumption (V2) is equivalent to the regularity \cite{DLM1}. That is, any weak module is completely
reducible.

We now recall the  notion of intertwining operators and fusion
rules from [FHL].
Let $W^i=(W^i,Y_{W^i})$ for $i=1,2,3$ be weak $V$-modules.
An intertwining
operator $\mathcal {Y}( \cdot , z)$ of type $\Itw{W^3}{W^1}{W^2}$ is a linear map
$$
\mathcal
{Y}(\cdot, z): W^1\rightarrow \Hom(W^2, W^3)\{z\}, \, v^1\mapsto
\mathcal {Y}(v^1, z) = \sum_{n\in \mathbb{C}}{v_n^1z^{-n-1}}
$$
satisfying the following conditions:
\begin{enumerate}
\item[(i)] For any $v^1\in W^1, v^2\in W^2$ and $\lambda \in \mathbb{C},
v_{n+\lambda}^1v^2 = 0$ for $n\in \mathbb{Z}$ sufficiently large.
\item[(ii)] For any $a \in V, v^1\in W^1$,
$$z_0^{-1}\delta(\frac{z_1-z_2}{z_0})Y_{W^3}(a, z_1)\mathcal
{Y}(v^1, z_2)-z_0^{-1}\delta(\frac{z_1-z_2}{-z_0})\mathcal{Y}(v^1,
z_2)Y_{W^2}(a, z_1)$$
$$=z_2^{-1}\delta(\frac{z_1-z_0}{z_2})\mathcal{Y}(Y_{W^1}(a, z_0)v^1, z_2).$$
\item[(iii)] For $v^1\in W^1$, $\dfrac{d}{dz}\mathcal{Y}(v^1,
z)=\mathcal{Y}(L(-1)v^1, z)$.
\end{enumerate}
The sum in the definition of intertwining operator defined in \cite{FHL} is over rational numbers. For a
rational vertex operator algebra, this is true. In general , the sum should be over complex numbers.
 All of the intertwining operators of type $\Itw{W^3}{W^1}{W^2}$ form a vector space denoted by
 $I_V\Itw{W^3}{W^1}{W^2}$. The dimension of $I_V\Itw{W^3}{W^1}{W^2}$ is called the
fusion rule of type $\Itw{W^3}{W^1}{W^2}$ for $V$, which is denoted by $N^{W^3}_{W^1, W^2}$.

The following properties of the fusion rule are well known (cf. \cite{FHL}).
\begin{prop}\label{2.1}
Let $V$ be a vertex operator algebra, and $M^i, M^j, M^k$ be three irreducible $V$-modules. Then we have:
\begin{enumerate}
\item[(i)] $N_{j, k}^{i}=N_{j, i^*}^{k^*}$, where we use $W^{i^*}$ to denote $(W^i)'$ and
    $N_{j,k}^i=N_{M^j,M^k}^{M^i};$
\item[(ii)] $N_{j,k}^{i}=N_{k,j}^{i}$.
\end{enumerate}
\end{prop}

Let $M^1$ and $M^2$ be two $V$-modules. A tensor product for the
ordered pair $(M^1, M^2)$ is a pair $(M, F(\cdot, z))$, which consists of
a $V$-module $M$ and an intertwining operator $F(\cdot, z)$ of type
$\Itw{M}{M^1}{M^2}$, such that the following universal property
holds: For any $V$-module $X$ and any intertwining operator $I(\cdot,
z)$ of type $\Itw{X}{M^1}{M^2}$, there exists a unique $V$-homomorphism $\phi$
from $M$ to $X$ such that $I(\cdot, z)=\phi\circ F(\cdot, z)$. Note that if
there is a tensor product, then it is unique by the universal mapping property.
In this case we will denote it by $M^1\boxtimes M^2$

In a series of papers \cite{HL1}, \cite{HL2}, \cite{HL3}, \cite{H1}, \cite{H2}, \cite{H3},  the tensor product
$\boxtimes$ of the modules for a vertex operator algebra $V$ has been defined and studied extensively.  We have
(cf. \cite[Cor. 10]{ABD} and \cite[Prop. 4.13]{HL1}) the following result.
\begin{thm}\label{2.2}
Let $V$ be a rational and $C_2$-cofinite vertex operator algebra, and $M^i, M^j, M^k$ be any three irreducible
modules of $V$. Then:
\begin{enumerate}
\item[(i)] The fusion rules $N_{i,j}^k$ are finite.
\item[(ii)] The tensor product $M^i\boxtimes M^j$ of $M^i$ and $M^j$ exists and
is equal to $\sum_{k}N_{i,j}^kM^k.$
\end{enumerate}
\end{thm}

We finally review some facts about the modular transformation of trace functions
 of irreducible modules of a vertex operator algebra from \cite{Z}. Let $V$ be a rational and $C_2$-cofinite
 vertex operator algebra, and $M^0,...,M^p$ be the irreducible $V$-modules as before. There is another
 VOA structure $(V,Y[\cdot,z],\bm1,\omega-c/24)$ on $V$ introduced in \cite{Z}. In particular,
 $$V=\oplus_{n\geq 0}V_{[n]}.$$
 We will write $wt [v]=n$ if $v\in  V_{[n]}.$ For each $v\in V_n$, we denote $v_{n-1}$ by $o(v)$ and
 extend to $V$ linearly. Recall that $M^i=\oplus_{n=0}^{\infty}M^i_{\lambda_i +n}.$ For $v\in V$ we set
 $$Z_i(v,q)=\tr_{M^i}o(v)q^{L(0)-c/24}=\sum_{n\geq 0}(\tr_{M^i_{\lambda_i+n}}o(v))q^{\lambda_i+n-c/24},$$
 which is a formal power series in variable $q.$ The constant $c$ here is the central charge of $V.$
 The $Z_i(\bm1,q)$ sometimes is called the $q$-character of $M^i.$ Then $Z_i(v, q)$ converges to a holomorphic
 function
 in $0<|q|<1$ \cite{Z}.
 As usual let $\h =\{\tau\in \BC \mid im\,\tau>0\}$ and
 $q=e^{2\pi i\tau}$ with $\tau\in \h.$ We also denote the holomorphic function $Z_i(v,q)$ by $Z_i(v,\tau)$ when
 we discuss modular transformations of these functions.

 The full modular group $SL_2{(\mathbb{Z})}$ acts on $\h$ by:$$\gamma: \tau\mapsto \frac{a\tau+b}{c\tau+d},\,
 \gamma=\mtx{a & b\\ c& d}
 \in SL_2(\mathbb{Z}).$$
The following theorem was established  in \cite{Z}.
\begin{thm}\label{2.3}
Let $V$ be a rational and $C_2$-cofinite vertex operator algebra, and let $M^0,...,M^p$ be the irreducible
$V$-modules. Then for any $\gamma\in SL_2(\Z)$ there exists $\rho_V(\gamma)=[\gamma_{ij}]_{i,j=0,...,p}\in
GL_{p+1}(\BC)$ such that for any $0\leq i\leq p$ and $v\in V_{[n]}$
$$Z_i(v,\gamma\tau)=({c\tau+d})^n\sum_{j=0}^p\gamma_{ij}Z_j(v,\tau).$$
 \end{thm}

Theorem \ref{2.3}, in fact, gives a group homomorphism $\rho_V: SL_2(\Z)\to GL_{p+1}(\BC).$ We call
$\rho_V(\gamma)$ the genus 1 modular matrices. In particular,
$$
S=\rho_V\left(\mtx{0 & -1\\ 1 & 0}\right) \text{ and }T=\rho_V\left(\mtx{1 & 1\\ 0 & 1}\right).
$$
are respectively called the \emph{genus one} $S$ and $T$-matrices of $V$. It is immediate to see that
$T_{jk}=\delta_{jk} e^{2 \pi i (\l_j -c/24)}$.

One of our main goals in this paper is to show that the kernel of $\rho_V$ is a congruence subgroup.

We also need the following results on the Verlinde formula \cite{Ver88} from \cite{H2} and \cite{H3} (also see
\cite{MS90}).
\begin{thm}\label{2.4}
Let $V$ be a vertex operator algebra satisfying (V1) and (V2). Then the genus one $S$-matrix of
$V$ defined above has the following properties:
\begin{enumerate}
 \item[(i)] $S$ is symmetric and $S^2=C$, where $C_{ij}=\delta_{ij^*}.$ In particular, $C$ has order at most 2 and is
     also symmetric.
\item[(ii)] $S^{-1}_{ij}=S_{i^*j}=S_{ij^*}.$
\item[(iii)] (Verlinde formula) For any $i,j,k\in\{0,...,p\}$
$$N_{i,j}^k =\sum_{q=0}^p\frac{S_{iq}S_{jq}S_{k^*q}}{S_{0q}}.$$
\end{enumerate}
\end{thm}

\subsection{Unitarity of $S$}\label{ss:3.2}
In this subsection, we will prove that the genus one $S$-matrix of $V$ defined in Section \ref{ss:3.1} is
unitary, and its consequences.  Our approach is slightly different from that given in \cite{ENO} for the unitarity of a normalized
$S$-matrix of a modular category with the following lemma. Recall that $S^2=C$ from Section \ref{ss:3.1}. In fact, this equality holds for any symmetric matrix satisfying the Verlinde formula as follows:

\begin{lem}\label{l:unitary}
  Let $\CC$ be a fusion category over $\BC$ with commutative Grothendieck ring. Suppose $A$ is a complex symmetric matrix indexed by $\Pi_\CC$ such that $A_{0r} \ne 0$ for all $r \in \Pi_\CC$ and $A$ satisfies the Verlinde formula in the sense that
  \begin{equation} \label{eq:Ver1}
    N_{ij}^k = \sum_{r \in \Pi_\CC} \frac{A_{ir}A_{jr}A_{k^*r}}{A_{0r}}
  \end{equation}
   for all $i,j,k \in \Pi_\CC$, where $N_{ij}^k$ is the fusion rule of $\CC$. Then, $A_{0r} \in \BR$, $A^2=C$ and  $A$ is unitary, where $C_{ij}=\delta_{ij^*}$ for $i, j \in \Pi_\CC$.
\end{lem}
\begin{proof}
By the Verlinde
formula  \eqref{eq:Ver1}, $\sum_{r \in \Pi_\CC} A_{ir}A_{jr} = N_{ij}^0 = \delta_{ij^*}$ for any $i, j\in \Pi_\CC$. This implies $A$ is invertible and $(A\inv)_{ij} = A_{ij^*} = A_{i^*j}$ for $i, j \in \Pi_\CC$. Hence, we have $A_{i^* j^*}=A_{ij}$ and $A_{0j} = A_{0j^*}$ for all $i,j \in \Pi_\CC$. Let $\KK_0(\CC)$ be the Grothendieck  ring of $\CC$ and let $\KK_\BC(\CC) =  \KK_0(\CC) \o_\BZ\BC$. Note that $\KK_\BC(\CC)$ is commutative $\BC$-algebra.   For $b \in \Pi_\CC$, let
$$
e_b = A_{0b} \sum_{a \in \Pi_\CC} A_{ab} \,a \in \KK_\BC(\CC),
$$
and $E=\{e_b \mid b\in \Pi_\CC\}$. Then
  \begin{eqnarray*}
    e_a e_b &=& A_{0a}A_{0b}\sum_{c,d} A_{ac} A_{bd} cd\\
    & =& A_{0a}A_{0b} \sum_{c,d,r} A_{ac} A_{bd}N_{cd}^r \,r\\
    &=&
   A_{0a}A_{0b} \sum_{c,d,r, z} A_{ac} A_{bd} \frac{A_{cz} A_{dz}A_{r^* z}}{A_{0z}} \,r \\
   & =&
   A_{0a}A_{0b}\sum_{r, z} \frac{\delta_{az^*}\delta_{bz^*} A_{r^*z}}{A_{0z}} r\\
   &=& \delta_{ab}A_{0a}^2\sum_{r} \frac{ A_{r^*a^*}}{A_{0a^*}} r \\
   & =&\delta_{ab}A_{0a}\sum_{r} A_{ra} r\\
   &=&\delta_{ab}e_a\,.
   \end{eqnarray*}
Hence, $E$ is the set of all primitive idempotents of $\KK_\BC(\CC)$.

 The duality permutation defined on $\Pi_\CC$ can be extended to a sesquilinear linear map $\dag$ on $\KK_\BC(\CC)$, i.e.
 $$
 (\sum_{x \in \Pi_\CC} \a_x x)^\dag = \sum_{x \in \Pi_\CC} \ol \a_x x^*
 $$
for $\a_x \in \BC$. Moreover, $\dag$ is an $\BR$-algebra automorphism of $\KK_\BC(\CC)$, but $\dag$ is not $\BC$-linear. In particular, $e_b^\dag \in E$ and hence $\dag$ defines a permutation on $E$.

For $x \in \KK_\BC(\CC)$, denote by $\e(x)$ the coefficient of the unit object $0$ in $x$. Then
   $$
   \e(ab) = N_{ab}^0 =\delta_{ab^*} \text{ for } a,b \in \Pi_\CC\,.
   $$
   We now define the sesquelinear form $(\cdot, \cdot)$ on $\KK_\BC(\CC)$ by
 $$
 (x, y) = \e(xy^\dag)\,.
 $$
 Note that $(x,x) > 0$ for $x \ne 0$.
 Thus
 $$
 0 < (e_b, e_b) = \e(e_b e_b^\dag)\,.
 $$
 Therefore, $e_b^\dag=e_b$ and so $(e_b, e_b)=A_{0b}^2 > 0$ and $A_{0b}A_{ab} = \ol{A_{0b}A_{a^*b}}$ for all $a, b \in \Pi_\CC$.  The former implies $A_{0b} \in \BR$ and hence $A_{ab} = \ol{A_{a^*b}}$ for all $a, b \in \Pi_\CC$. Therefore, $A$ is unitary.
\end{proof}

The following corollary is an immediate consequence of Lemma \ref{l:unitary} and the modularity of $\CC_V$ presented in Theorem \ref{4.4}.
\begin{cor}\label{3.5}
 Let $V$ be a vertex operator algebra satisfying  (V1) and (V2). Then the genus one $S$-matrix of
 $V$ defined in Section \ref{ss:3.1} is unitary and satisfies $\bar{S}=SC.$
 \end{cor}

The following result can be proved easily by using Corollary \ref{3.5}:
 \begin{cor}\label{3.6}
 Let $V$ be a vertex operator algebra satisfying  (V1) and (V2).
 For any $u\in V_{[m]},$ $v\in V_{[n]},$  $\gamma=\mtx{a & b\\ c& d}\in SL_2(\Z)$ and $\tau_1,\tau_2\in \h$  we
 have
 $$\sum_{i}Z_i(u,\gamma\tau_1)\overline{Z_i(v,\gamma\tau_2)}=
 (c\tau_1+d)^m\overline{(c\tau_2+d)^n}\sum_{i}Z_i(u,\tau_1)\overline{Z_i(v,\tau_2)}.$$
 In particular, $\sum_{0\leq i\leq p}|\chi_i(\tau)|^2$ is invariant under the action of $SL_2(\Z).$
 \end{cor}
 \begin{proof}
 Note that $T$ is a diagonal matrix with diagonal entries $e^{2\pi i(\lambda_j-c/24)}$ for $j=0,...,p$ which
 is clearly a unitary matrix as $\lambda_j$ and $c$ are rational numbers.  It follows from
  Corollary \ref{3.5} that the representation $\rho$ is unitary. Set
  $$f(\tau_1,\tau_2)=\sum_{i}Z_i(u,\tau_1)\overline{Z_i(v,\tau_2)}.$$
  \begin{eqnarray*}
  f(\gamma\tau_1,\gamma\tau_2)&=&\sum_{i}Z_i(u,\gamma\tau_1)\overline{Z_i(v,\gamma\tau_2)}\\
  &=&(c\tau_1+d)^m\overline{(c\tau_2+d)^n}\sum_{i,j,k}\gamma_{ij}Z_j(u,\tau_1)\overline{\gamma_{ik}}\overline{Z_k(v,\tau_2)}\\
  & =&(c\tau_1+d)^m\overline{(c\tau_2+d)^n}\sum_{i}Z_i(u,\tau_1)\overline{Z_i(v,\tau_2)},
  \end{eqnarray*}
as required.
 \end{proof}

Here we use Corollary \ref{3.6} to study the extensions of vertex operator algebras. As before we assume that
$V$ is a vertex operator algebra satisfying  (V1) and (V2).  We also assume that $U$ is an extension
of $V$ satisfying (V1) and (V2). Then $U=\sum_{i}n_iM^i$ as a $V$-module, where $n_i\geq 0$ and $n_0=1$, since the
vacuum vector is unique. The main goal is to determine the possible values of $n_i.$ There have been a lot of
discussions
on this in the literature (see for example, \cite{CIZ1}-\cite{CIZ2} and \cite{G1}) using the modular invariance
of the characters. It seems that using the characters of irreducible modules is not good enough as the
characters of irreducible modules are not linearly independent in general. In this section we use the conformal
blocks instead of the characters to approach the problem.

For $u,v\in V$, we set
$$f_V(u,v,\tau_1,\tau_2)=\sum_{i=0}^pZ_i(u,\tau_1)\overline{Z_i(v,\tau_2)}$$
(cf. Corollary \ref{3.6}). Similarly we can define
$$f_U(u,v,\tau_1,\tau_2)=\sum_{M}Z_M(u,\tau_1)\overline{Z_M(v,\tau_2)}$$
for $u,v\in U$ where $M$ ranges through the equivalent classes of irreducible $U$-modules.
Since each irreducible $U$-module $M$ is a direct sum of irreducible $V$-modules, we see that for $u,v\in V$
$$f_U(u,v,\tau_1,\tau_2)=\sum_{i,j=0}^pX_{ij}Z_i(u,\tau_1)\overline{Z_j(v,\tau_2)}$$
for some $X_{ij}\in\Z_+$ for all $i,j.$ If $u=v=\bm1$ and $\tau_1=\tau_2=\tau,$ then
$f_U(\bm1,\bm1,\tau,\tau)$, which is the sum of square norms of the irreducible characters of $U$, is
$SL_2(\Z)$-invariant. We now determine the matrix $X=[X_{ij}].$ It will be clear from our proof below that
the $SL_2(\Z)$-invariance of  $f_U(\bm1,\bm1,\tau,\tau)$ is not good enough to determine the matrix
$X.$

\begin{prop} The matrix $X$ satisfies \emph{(i)} $X_{00}=1,$ \emph{(ii)} $X\gamma=\gamma X$ where $\gamma\in
SL_2(\Z)$ and is identified with the modular transformation matrix $\rho_V(\gamma).$
\end{prop}
 \begin{proof}
For any $u\in V_{[m]}$, let ${\bf Z}(u,\tau)=\mtx{ Z_0(u,\tau)\\ \vdots \\
Z_p(u,\tau)}.$ Then
$$
{\bf Z}(u,\gamma\tau)=(c\tau+d)^m\gamma  {\bf Z}(u,\tau) \quad\text{and}\quad
f_U(u,v,\tau_1,\tau_2)={\bf Z}(u,\tau_1)^TX\overline{{\bf Z}(v,\tau_2)}.
$$
 By Corollary \ref{3.6},
\begin{eqnarray*}
&&
 (c\tau_1+d)^m\overline{(c\tau_2+d)^n}{\bf Z}(u,\tau_1)^TX\overline{{\bf Z}(v,\tau_2)}\\
&=&f_U(u,v,\gamma\tau_1,\gamma\tau_2)\\
&=&{\bf Z}(u,\gamma\tau_1)^TX\overline{{\bf Z}(v,\gamma\tau_2)}\\
&=& (c\tau_1+d)^m\overline{(c\tau_2+d)^n}{\bf Z}(u,\tau_1)^T\gamma^TX\bar\gamma\overline{{\bf Z}(v,\tau_2)}.
\end{eqnarray*}
This implies that
$${\bf Z}(u,\tau_1)^TX\overline{{\bf Z}(v,\tau_2)}={\bf Z}(u,\tau_1)^T\gamma^TX\bar\gamma\overline{{\bf
Z}(v,\tau_2)}$$
for all $u,v.$ Since $\gamma$ is unitary, it is enough to show that if ${\bf Z}(u,\tau_1)^TA\overline{{\bf
Z}(v,\tau_2)}=0$ for all $u,v\in V$ where $A=[a_{ij}]$ is a fixed matrix, then $A=0.$

Note that ${\bf Z}(u,\tau_1)^TA\overline{{\bf
Z}(v,\tau_2)}=\sum_{ij}a_{ij}Z_i(u,\tau_1)\overline{Z_j(v,\tau_2)}.$
For simplicity, we set $q_j=e^{2\pi i\tau_j}$ for $j=1,2.$ Then
\begin{eqnarray*}
& &0={\bf Z}(u,\tau_1)^TA\overline{{\bf Z}(v,\tau_2)}\\
& &\ \ \ =\sum_{i,j}\sum_{m_i,n_j\geq
0}a_{ij}(\tr_{M^i_{\lambda_i+m_i}}o(u)\overline{\tr_{M^j_{\lambda_j+n_j}}o(v)})q_1^{\lambda_i+m_i-c/24}\overline{q_2^{\lambda_j+n_j-c/24}}.
\end{eqnarray*}
This implies that each coefficient of $q_1^m\overline{q_2^n}$ for any rational numbers $m,n$ must be zero. We
now
prove that $a_{ij}=0$ for all $i,j.$ Fix $i,j.$ Then the coefficient
 of $q_1^{\lambda_i-c/24}\bar q_2^{\lambda_j-c/24}$ in ${\bf Z}(u,\tau_1)^TA\overline{{\bf Z}(v,\tau_2)}$ is
$$\sum_{k,l}a_{kl}\tr_{M^k_{\lambda_k+m_k}}o(u)\overline{\tr_{M^l_{\lambda_l+n_l}}o(v)}$$
where $k,l\in\{0,...,p\}$ satisfy $m_k+\lambda_k=\lambda_i,  n_l+\lambda_l=\lambda_j.$
Fix $n\geq 0$ such that $n\geq m_k,n_l$ for all $k,l$ occurring in the summation above. Recall from
\cite{DLM3} that there is a finite dimensional semisimple associative algebra $A_n(V)$ such that
$M^k_{m_k+\lambda_k}, M^l_{n_l+\lambda_l}$ are the inequivalent simple modules of $A_n(V).$ As a result we can
choose $u,v\in V$ such that $o(u)=1$ on $M^i_{\lambda_i}$ and $o(u)=0$ on
all other $M^k_{\lambda_k+m_k},$ $o(v)=1$ on $M^j_{\lambda_j}$ and $o(v)=0$ on all other $M^l_{\lambda_l+n_l}.$
Thus, for this $u$ and $v,$ the coefficient of $q_1^{\lambda_i-c/24}\bar q_2^{\lambda_j-c/24}$ in ${\bf
Z}(u,\tau_1)^TA\overline{{\bf Z}(v,\tau_2)}$
is a nonzero multiple of $a_{ij}.$ This forces $a_{ij}=0.$ The proof is complete.
\end{proof}
\subsection{The congruence property theorem}

Now we come back to the theories of vertex operator algebras. Let $V$ be a rational and $C_2$-cofinite vertex
operator algebra. For any $V$-module $M$, set $\theta_{M}=e^{2\pi i L(0)}.$ The following result from
\cite[Theorem 4.1]{H3} is important in this paper.
\begin{thm}\label{4.4}
 Let $V$ be a vertex operator algebra satisfying  (V1) and (V2). Then the $V$-module category $\CC_V$
 with the dual $M'$ ($M$ a $V$-module),  braiding $\sigma$ which is denoted by
 $\CC$ in \cite[Page 877]{H3}  and twist $\theta$ is a modular tensor category over $\BC.$
 \end{thm}

Note that $\End_V(M^i)=\mathbb{C}, 0\leq i\leq p.$ Recall from discussion in Sections \ref{ss:sphcat} and
\ref{ss:MTC} that the pivotal dimension $d_i$ of the simple $V$-module is a non-zero real number, and the global
dimension $\dim \CC_V = \sum_{i=0}^p d_i^2 \ge 1$. We let $\ts$ and $\tt$ be the $S$ and $T$-matrices of
$\CC_V$, and $D = \sqrt{\dim \CC_V}$, the positive square root of $\dim \CC_V$, and $\bc$ the central charge of
$\CC_V$. We fix the normalization $s=\frac{1}{D} \ts$, and simply call $s$ \emph{the normalized  $S$-matrix of
$\CC_V$}. We will prove in Theorem \ref{iden} that $s$ is identical to the genus one $S$-matrix of $V$ up to a sign.

\begin{thm}\label{iden}
 Let $V$ be a vertex operator algebra satisfying  (V1) and (V2). Then
 \begin{enumerate}
   \item[(i)] The normalized $S$-matrix $s$ of $\CC_V$ and  the genus one $S$-matrix of $V$ are identical up to a sign.

   \item[(ii)] The representation $\rho_V$ defined by modular transformation of trace functions is a modular
       representation of $\CC_V$. In particular, $\ker \rho_V$ is a  congruence subgroup of level $n$ where
       $n$ is the order of the genus one $T$-matrix of $V$, and $\rho_V$ is $\BQ_n$-rational.
       \item[(iii)] The central charge $\bc$ of $\CC_V$ is equal to the central charge $c$ of $V$ modulo
           $4$.
 \end{enumerate}
\end{thm}
\begin{proof}
Let
$$\sigma_{M^iM^j}: M^i\boxtimes M^j\to M^j\boxtimes M^i$$
be the braiding of $\CC_V$. It is proved in \cite{H3} that the pivotal trace of
$\sigma_{M^{i^*}M^j}\sigma_{M^jM^{i^*}}$ on $ M^j\boxtimes M^{i^*}$  equals  $\frac{S_{ij}}{S_{00}}.$
This implies that $S=\lambda s$ where $\lambda=\frac{S_{00}}{s_{00}}.$  Using
the unitarity of $s$ and $S$, we conclude that $\lambda$ is a complex number  of norm $1.$ This
forces $\lambda=\pm 1$, which proves the first statement.

It follows from Theorem \ref{4.4} that the $T$-matrix of $\CC_V$ is given by $\tt=[\delta_{ij} \theta_i]_{i,j=0,..., p}$ and $\theta_j =
e^{2 \pi i \l_j}$. Therefore,  that genus one $T$-matrix of $V$
is given by $T = \tt \,e^{-2\pi ic/24}$, where $c$ is the central charge of $V$. In particular,
$\rho_V$ is a modular representation of $\CC_V$. The second part of the second statement is an immediate
consequence of Theorem II (i) and (ii).

 By (i), \eqref{eq:STrelations} and  Theorem \ref{2.4} we see that
  $$
  C=(ST)^3=\pm (s\tt e^{-2\pi ic/24})^3
 =\pm \frac{p^+}{D}e^{-6\pi ic/24} C\,,
  $$
  where $p^+$ is the Gauss sum of $\CC_V$. This implies that $\pm 1 = \frac{p^+}{D} e^{-\pi ic/4}$ or
  $\frac{p^+}{D}=\pm e^{\pi ic/4}.$ In particular, $\bc = c \pmod 4$.
\end{proof}

 {\rm \bf Theorem I} now follows from Theorem \ref{iden} immediately.

We next discuss two different definitions of dimension of modules of rational and $C_2$-cofinite vertex
operator algebras given in \cite{DJX} and \cite{BaKi}. As before we assume that $V$ is a vertex operator
algebra satisfying (V1) and (V2). Recall the following definition of quantum dimension from
\cite{DJX}.
Let  $M$  be a $V$-module. Set $Z_M(\tau)=ch_qM=Z_M(\bm1,\tau).$  The quantum dimension of $M$ over $V$ is
defined as
$$\qdim_{V}M=\lim_{y\rightarrow 0}\frac{Z_M(iy)}{Z_V(iy)}$$
where $y$ is real and positive.
It is shown in \cite{DJX} that
if $V$ is a vertex operator algebra satisfying (V1) and (V2) with the irreducibles  $M^i$ for
$i=0,..., p$ such that $\lambda_i>0$ for $i\ne 0,$ then
\begin{equation}\label{eq:qdim}
\qdim_{V} M^i=\frac{S_{i0}}{S_{00}}.
\end{equation}

On the other hand, because $V$ is a vertex operator algebra satisfying (V1) and (V2), the tensor
category $\CC_V$ of $V$-modules is  modular  by Theorem \ref{4.4}. The pivotal dimension $d_i=\dim M^i$ of
$M^i$ is also defined in the modular tensor category $\CC_V$. We now prove that
these two dimensions coincide.
\begin{prop}\label{iden dim}
Let $V$ be a vertex operator algebra satisfying (V1) and (V2), and $\lambda_i>0$ for $i\ne 0.$
Then for any irreducible
$V$-module $M^i$, $\dim M^i=\qdim_V M^i.$
\end{prop}
\begin{proof}
 Since $\dim M^i=d_i=\frac{{s}_{0i}}{s_{00}}$, the result follows
 from Theorem \ref{iden} and \eqref{eq:qdim} immediately.
 \end{proof}

The modular transformation property on the conformal blocks has been used extensively in the study of rational
vertex operator algebras. The modular transformation property gives an estimation of the growth conditions on
the dimensions
of homogeneous subspaces as the $q$-character of an irreducible module is a component of a vector valued modular
function \cite{KM}. The growth condition helps us to show that a rational and $C_2$-cofinite vertex operator
algebra with central charge less than one is an extension of the Virasoro vertex operator algebra associated to
the discrete series \cite{DZ} and to characterize vertex operator algebra $L(1/2,0)\otimes L(1/2,0)$ \cite{ZD},
\cite{DJ1}. The congruence subgroup property of the action of the modular group on the conformal block is
expected to play an important role in the classification of rational vertex operator algebras. Since the
$q$-character of an irreducible module is a modular function on a congruence subgroup and the sum of the square
norms of the $q$-characters of the irreducible modules is invariant under $SL_2(\Z)$, this gives a lot of
information on the dimensions of homogeneous
subspaces of vertex operator algebras. For example, one can use these properties to determine the possible
characters of the rational vertex operator algebras of central charge 1 \cite{Ki}. This will avoid some
difficult work in \cite{DJ2} and \cite{DJ3} of determining the dimensions of homogenous subspaces of small
weights when characterizing certain classes of rational vertex operator algebras of central charge one.

\section{Galois Symmetry of Modular Representations}\label{s:Gal1}
It was conjectured by Coste and Gannon that the representation of $\SL{}$ associated with a RCFT admits
a
Galois symmetry (cf. \cite[Conj. 3]{CG99} and \cite[6.1.7]{GanBook}). Under certain assumptions, the
Galois symmetry of
these representations of $\SL{}$ was established by Coste and Gannon \cite{CG99} and by Bantay
\cite{Bantay03}.

In this section, we will prove that such Galois symmetry holds for all modular representations
of a modular category as stated in Theorem II (iii) and (iv). It will follow from  Theorem \ref{iden} that this
Galois symmetry holds for the representation $\rho_V$ defined by modular transformation of the trace functions
of any VOA $V$ satisfying  (V1) and (V2).

The Galois symmetry for the canonical
modular
representation of the Drinfeld center of a spherical fusion category (Lemma \ref{l:double case}) plays a
crucial for the general case, and we will provide its proof in the next section.
\subsection{Galois action on a normalized $S$-matrix}
Let $\A$ be a  modular category over $\k$ with Frobenius-Schur exponent $N$, and $\rho$ a level $n$
modular
representation of $\A$.
By virtue of Theorem II (i) and (ii), $N \mid n \mid 12 N$ and $\rho(\SL{}) \le
\GLR{\Pi}{\BQ_n}$, where $\Pi_\A$ is simply abbreviated as $\Pi$.

For a fixed 6-th root $\zeta$ of the anomaly of $\A$, $\zeta$ determines the modular representation
$\rho^\zeta$  of $\A$ (cf. \eqref{eq:repC}). It follows from Section \ref{ss:SL2} that $\rho =
\rho_x^\zeta$
for some 12-th root of unity $x \in \k$. Let
$$
s = \rho(\fs) \quad \text{and}\quad t = \rho(\ft).
$$
Then
\begin{equation}\label{eq:normalization}
s = \frac{\zeta^3}{x^3 p_\A^+} \ts, \quad t = \frac{x}{\zeta} \tt\quad \in \GLR{\Pi}{\BQ_n}\,.
\end{equation}
Thus $s^2 = x^6 C = \pm C$,  where $C$ is the charge conjugation matrix
$[\delta_{ij^*}]_{i,j \in \Pi}$. Set $\sgns = x^6$.

Following \cite[App. B]{BG91}, \cite{CG} or \cite[App.]{ENO}, for each $\s \in \AQab$, there exists
a unique permutation, denoted by  $\hs$, on $\Pi$ such that
\begin{equation}\label{eq:galois1}
  \s\left(\frac{s_{ij}}{s_{0j}}\right) = \frac{s_{i\hs(j)}}{s_{0\hs(j)}}    \quad \text{for all }i, j\in
  \Pi.
\end{equation}
Moreover, there exists a function $\e_\s : \Pi \to \{\pm 1\}$ such that
\begin{equation}\label{eq:galois2}
 \s(s_{ij}) = \e_{\s}(i) s_{\hs(i) j} = \e_{\s}(j) s_{i \hs(j)} \quad \text{for all }i, j \in \Pi\,.
\end{equation}

Define $\gs \in \GLR{\Pi}{\BZ}$   by $(\gs)_{ij}= \e_{\s}(i) \delta_{\hs(i) j}$. Then
\eqref{eq:galois2} can be rewritten as
\begin{equation}\label{eq:galois3}
  \s(s)= \gs s = s \gs\inv
\end{equation}
where $(\s(y))_{ij}=\s(y_{ij})$ for $y \in \GLR{\Pi}{\BQ_n}$. Since $G_\s \in \GLR{\Pi}{\BZ}$,
this equation implies that the assignment,
$$
\AQab \to\GLR{\Pi}{\BZ}, \s \mapsto \gs
$$
defines a representation of the  group $\AQab$ (cf. \cite{CG}). Moreover,
\begin{equation} \label{eq:first conjuation}
  \s^2(s) = \gs s \gs\inv,
\end{equation}
\begin{equation} \label{eq:Gs formula}
\gs =  \s(s) s\inv = \s(s\inv) s\,.
\end{equation}
Note that the permutation $\hs$ on $\Pi$ depends only on the modular category $\A$ as
$\frac{s_{ij}}{s_{0j}} = \frac{\ts_{ij}}{\ts_{0j}}$ in \eqref{eq:galois1}. However, the matrix $\gs$ does
depend on $s$, and hence the representation $\rho$.

Suppose  $\tt=[\delta_{ij} \w_j]_{i, j\in \Pi}$.  Then $t=\frac{x}{\zeta} \tt$ is a diagonal matrix of order
$n$. If $\s|_{\BQ_n} = \s_a$ for some integer $a$ relatively prime to $n$, then
 $$
 \s(t) = \s_a(t)= t^a\,.
 $$
By virtue of \eqref{eq:first conjuation}, to prove Theorem II (iii), it suffices to
show that
\begin{equation}\label{eq:CG-conjecture}
 \s^2(t) = \gs t \gs\inv.
\end{equation}
We first establish the following simple observation.
\begin{lem}\label{l:iha1}
   For any integers $a, b$  such that $ab \equiv 1 \mod n$, we have
  $$
    s^2 = ( t^a  s  t^b  s  t^a)^2\,.
  $$
\end{lem}
\begin{proof}
  It follows from direct computation that
   $$
   \fs^2 \equiv \mtx{0 & -a \\ b & 0}^2\equiv (\ft^a \fs \ft^b \fs \ft^a)^2 \mod n\,.
   $$
   By Theorem II (i), $\rho$ factors through $\SL{n}$ and so we obtain the equality.
\end{proof}
\subsection{Galois symmetry of Drinfeld doubles} Before we return to prove the Galois symmetry for general
modular categories, we need to settle the special case, stated in the following lemma,  when $\A$ is the
Drinfeld center of a
spherical fusion category over $\k$, and $\rho$ is the canonical modular representation of $\A$.

\begin{lem}\label{l:double case}
Let $\CC$ be a spherical fusion category over $\k$, and $\s \in \AQab$. Suppose $\gs$ is the signed
permutation matrix determined by the canonical normalization $s=\frac{1}{\dim \CC} \ts$ of the
$S$-matrix of the center $Z(\CC)$, i.e. $\gs = \s(s)s\inv$. Then the $T$-matrix $\tt$ of $Z(\CC)$ satisfies
\begin{equation}\label{eq:T-symmetry}
\s^2(\tt) = \gs \tt \gs\inv.
\end{equation}
In particular, if $(\gs)_{ij}= \es(i) \delta_{\hs(i)j}$ for some sign function $\es$ and permutation $\hs$ on $\Pi_{Z(\CC)}$, then $\s^2(\tt_{ii}) = \tt_{\hs(i)\hs(i)}$ for all $i \in \Pi_{Z(\CC)}$.
Moreover, for any integers $a, b$ relatively prime to $N=\ord(\tt)$ such that $\s|_{\BQ_N} = \s_a$ and $ab
\equiv
1
\mod N$,
  $$
  \gs=\tt^a s \tt^b s \tt^a s\inv\,.
  $$
\end{lem}
The proof of this lemma, which requires the machinery of generalized Frobenius-Schur indicators, will be
developed independently in Section \ref{s:Quantum_double}.
\subsection{Galois symmetry of general modular categories} Let $c$ be the braiding of
the modular category $\A$. Without loss of generality, we further assume the underlying pivotal category of
$\A$ is \emph{strict}.
We set
$$
\s_{X \o Y}(V) = (c_{X,V} \o Y)\circ (X \o c\inv_{V,Y})
$$
for any $X, Y, V \in \A$. Then $(X \o Y, \s_{X \o Y})$ is a simple object of $Z(\A)$ if $X, Y$ are
simple
objects of $\A$. Moreover, if $V_i$ denotes a representative of $i \in \Pi$, then
$$
\{(V_i \o V_j, \s_{V_i \o V_j})\mid i, j \in \Pi\}
$$
forms a complete set of representatives of simple objects in $Z(\A)$ (cf. \cite[Sect. 7]{MugerII03}). Let $(i,j) \in \Pi \times
\Pi$ denote the
isomorphism class of $(V_i \o V_j, \s_{V_i \o V_j})$ in $Z(\A)$. Then we have $\Pi_{Z(\A)} = \Pi
\times
\Pi$ and the isomorphism class of the unit object of $Z(\A)$ is $(0,0) \in \Pi_{Z(\A)}$.

Let  $\ts$ and $\tt=[\delta_{ij} \w_i]_{i,j \in \Pi}$ be the $S$ and $T$-matrices
of
$\A$  respectively. Then  the $S$ and $T$-matrices of the center $Z(\A)$, denoted by $\BS$ and $\BT$
respectively, are indexed by $\Pi \times \Pi$. By \cite[Sect. 6]{NS4},
  $$
  \BS_{ij, kl} = \ts_{ik}\ts_{j l^*} ,\quad \BT_{ij,kl}= \delta_{ik}\delta_{jl} \frac{\w_i}{\w_j}\,.
  $$
  Thus  $\FSexp(\A) = \ord(\BT) = \ord(\tt)=N$.
\begin{proof}[Proof of Theorem \emph{II (iii)} and \emph{(iv)}]
 The canonical normalization $\bs$ of $\BS$ is
  $$
  \bs_{ij, kl} =  \frac{1}{\dim \A} \ts_{ik} \ts_{j l^*} = \sgns s_{ik} s_{j l^*},
  $$
  where $\sgns = \pm 1$ is given by $s^2 = \sgns C$ (cf. \eqref{eq:normalization}). Moreover, $\bs \in
  \GLR{\Pi\times
  \Pi}{\BQ_N}$.

  For $\s \in \AQab$, we have
  $$
  \s(\bs_{ij, kl}) = \sgns \e_\s(i)\e_\s(j) s_{\hs(i) k} s_{\hs(j) l^*}  =\e_\s(i)\e_\s(j) \bs_{
  \hs(i)
  \hs(j), kl} = \be_\s(i,j) \bs_{\bhs(i,j), kl}\,,
  $$
  where $\be_\s$ and $\bhs$ are respectively the associated sign function and permutation on $\Pi
  \times \Pi$.
  Thus,
  $$
  \be_\s(i,j)=\e_\s(i)\e_\s(j), \quad \bhs(i,j)=(\hs(i),\hs(j))
  $$
  and so
  $$
  (\boldsymbol G_\s)_{ij, kl} = \e_\s(i)\e_\s(j) \delta_{\hs(i)k}\delta_{\hs(j)l}
  $$
  where $\boldsymbol G_\s$ is the associated signed permutation matrix of $\s$ on $\bs$.
  By Lemma \ref{l:double case}, we find
  $$
  \s^2\left(\frac{\w_i}{\w_j}\right)= \s^2(\BT_{ij, ij}) = \BT_{\bhs(i,j), \bhs(i,j)} =
  \BT_{\hs(i)\hs(j), \hs(i) \hs(j)} =   \frac{\w_{\hs(i)}}{\w_{\hs(j)}}
  $$
  for all $i, j \in \Pi$. Since $\w_0 =1$,
  $$
  \frac{\w_{\hs(i)}}{\s^2(\w_i)}=\frac{\w_{\hs(0)}}{\s^2(\w_0)} = \w_{\hs(0)}
  $$
   for all $i \in \Pi$. By \eqref{eq:normalization},  $t = \tilde\zeta\inv \tt$ where $\tilde\zeta =
   \zeta/x
   $. Then,
  \begin{equation}\label{eq:form 1}
    t_{\hs(i)\hs(i)} =\frac{\w_{\hs(i)}}{\tilde\zeta} =\frac{\s^2(\w_i)
    \w_{\hs(0)}}{\tilde\zeta}=\s^2(t_{ii})
  \b
  \end{equation}
  for all $i \in \Pi$, where  $\b= t_{\hs(0)\hs(0)}\cdot \s^2(\tilde\zeta) \in \k^\times$. Suppose
  $\s|_{\BQ_n}=\s_a$
  for some integer $a$ relatively prime to $n$.  Then \eqref{eq:form 1} is equivalent to the equalities
  \begin{equation}\label{eq:diff forms}
   \gs t \gs\inv =\b t^{a^2} \quad\text{or} \quad \gs\inv t^{a^2} \gs =\b\inv t\,.
  \end{equation}
  Now, it suffices to show that $\b=1$.

  Apply $\s^2$ to the equation $(s\inv t)^3 = \id$. It follows from \eqref{eq:diff forms} that
   $$
     \id =  \gs s\inv \gs\inv t^{a^2} \gs s\inv \gs\inv t^{a^2}   \gs s\inv \gs\inv t^{a^2} =
          \b^{-2} (\gs s\inv t s\inv t s\inv \gs\inv t^{a^2})\,.
   $$
   This implies
   $$
   \id =  \b^{-2} (s\inv t s\inv t s\inv \gs\inv t^{a^2} \gs) = \b^{-3} (s\inv t s\inv t s\inv
   t)=\b^{-3}
   \id.
   $$
   Therefore, $\b^3=1$.

   Apply $\s\inv$ to the equality $sts=t\inv s t\inv$. Since $\s\inv|_{\BQ_n} = \s_b$ where $b$ is an
   inverse of $a$ modulo $n$, we have
   $$
   \gs\inv s  t^b s \gs = t^{-b} s \gs  t^{-b} \quad\text{or}\quad
    s  t^b s  = \gs t^{-b} s \gs  t^{-b} \gs\inv\,.
   $$
   This implies
   \begin{multline*}
    \gs\inv t^{a} s  t^b  s  t^{a} \gs = \gs\inv t^{a} \gs t^{-b} s\gs  t^{-b} \gs\inv t^{a} \gs \\ =
    \s\inv(\gs\inv t^{a^2} \gs) t^{-b} s \gs  t^{-b} \s\inv(\gs\inv t^{a^2} \gs) \\
    = \s\inv(\b\inv) t^b t^{-b} s \gs t^{-b}  \s\inv(\b\inv) t^b=\s\inv(\b^{-2}) s \gs\,.
   \end{multline*}
   Therefore,
  \begin{equation} \label{eq:zgs}
   t^{a} s  t^b  s  t^{a} =  \s\inv(\b^{-2}) \gs s  \,.
  \end{equation}
   Note that
   $$
   (\gs s)^2 = \gs s \gs s = s \gs\inv \gs s = s^2\,.
   $$
   Square both sides of \eqref{eq:zgs} and apply Lemma \ref{l:iha1}. We obtain
   $$
     s^2 =  \s\inv(\b^{-4}) s^2\,.
   $$
   Consequently, $\s\inv(\b^{-4})=1$ and this is equivalent to $\b^4=1$. Now, we can conclude that
   $\b=1$
   and so
   $$
   \gs t \gs\inv = t^{a^2} \,.
   $$
   By \eqref{eq:zgs}, we also have $\gs= t^a s t^{b} s t^a s\inv$.
\end{proof}
\begin{remark}
  \rm{For the case $\A = \Rep(D(H))$, where $H$ is a semisimple Hopf algebra, the $T$-matrix $\tilde t$ of $\A$
  was proven to satisfy $\s^2(\tilde t_{ii}) = \tilde t_{\hs(i)\hs(i)}$ in \cite[Prop. 12.1]{SZh}. The underlying
  modular representation of $\A$, in the context of Theorem II (iii) and (iv), is the canonical modular
  representation of $\A$ described in Section \ref{ss:center}.}
\end{remark}
We can now establish the Galois symmetry of RCFT as a corollary.
\begin{cor}
  Let $V$ be a vertex operator algebra satisfying (V1) and (V2) with simple $V$-modules $M^0, \dots,
  M^p$. Then the genus one $S$ and $T$ matrices of $V$ admit the Galois symmetry: For $\s \in \AQab$,  there
  exists a signed permutation matrix $\gs \in \GLC{p+1}$ such that
  $$
  \s(S) = \gs S = S \gs\quad \text{and}\quad \s^2(T) = \gs T \gs\inv
  $$
  where the associated permutation $\hs \in S_{p+1}$ of $\gs$ is determined by
  $$
  \s\left(\frac{S_{ij}}{S_{0j}}\right) = \frac{S_{i\hs(j)}}{S_{0\hs(j)}}\quad \text{for all }i,j=0,\dots,p.
  $$
  In particular, $\s^2(T_{ii}) = T_{\hs(i)\hs(i)}$.
  If $n=\ord(T)$ and $\s|_{\BQ_n} = \s_a$ for some integer $a$ relatively prime to $n$, then
  $$
  \gs = T^a S T^b S T^a S\inv
  $$
  where $b$ is an inverse of $a$ modulo $n$.
\end{cor}
\begin{proof}
  The result is an immediate consequence of  Theorem \ref{iden} and Theorem II (iii) and (iv).
\end{proof}

\begin{remark}
  {\rm
  The modular representation $\rho$ factors through a representation $\rho_n : \SL{n} \to \GLk{\Pi}$. For
  any integers $a, b$ such that $ab \equiv 1 \mod n$, the matrix
  $$
  \fd_a = \mtx{a & 0\\ 0 & b} \equiv \ft^a \fs \ft^b \fs \ft^a \fs\inv \mod n
  $$
  is uniquely determined in $\SL{n}$ by the coset $a+n\BZ$ of $\BZ$. Moreover, the assignment $u: \GalQ{n} \to
  \SL{n}$, $\s_a \mapsto \fd_a$, defines a group monomorphism.
  Theorem II (iv) implies that the representation $\phi_\rho:\GalQ{n} \to \GLR{\Pi}{\BZ}, \s \mapsto \gs,$
  associated with $\rho$, also factors through $\rho_n$ and they satisfy the commutative diagram:
  $$
  \xymatrix{ \GalQ{n} \ar[r]^-{\phi_\rho} \ar[d]_-{u} & \GLk{\Pi} & \\
  \SL{n} \ar[ru]^-{\rho_n}& \SL{} \ar[l]^-{\pi_n} \ar[u]^-{\rho}\,.  &  }
  $$
  }
\end{remark}

The Galois symmetry enjoyed by the $T$-matrix of the Drinfeld center of a spherical fusion category
(Lemma
\ref{l:double case}) does not hold for a general modular category, as demonstrated in the following
example.
\begin{example}\label{ex:fail T sym}{\rm
  Consider the Fibonacci modular category $\A$ over $\BC$ which has only one isomorphism class of
  non-unit
  simple objects, and we abbreviate this non-unit class by $1$ (cf. \cite[5.3.2]{RSW}).  Thus,
  $\Pi_\A=\{0,1\}$. The $S$ and $T$-matrices are given by
  $$
 \ts=\mtx{1 & \varphi \\ \varphi & -1}, \quad \tt=\mtx{1 & 0 \\ 0 & e^{\frac{4\pi i}{5}}}\,.
  $$
  where $\varphi=\frac{1+\sqrt{5}}{2}$. The central charge $\bc=\frac{14}{5}$ and $\dim \A = 2+\varphi$.
  Therefore, $\a = e^{\frac{7\pi i}{5}}$ is the anomaly of $\A$ and $\zeta=e^{\frac{7 \pi i }{30}}$ is a
  6-th root of $\a$ (cf. \eqref{eq:charge}). Thus
  $$
  s=\rho^{\zeta}(\fs) = \frac{1}{\sqrt{2 +\varphi}} \ts,\quad t=\rho^{\zeta}(\ft)  = \mtx{e^{\frac{-7\pi i
  }{30}} & 0 \\ 0 & e^{\frac{17 \pi i}{30}}}
  $$
  and so $\rho^\zeta$ is a level 60 modular representation of $\A$ by Theorem II. In $\GalQ{60}$,
  $\s_{49}$
  is the unique non-trivial square. Since $\s_7(\sqrt{5}) = -\sqrt{5}$,
  $\s_7\left(\frac{\ts_{i0}}{\ts_{00}}\right) = \frac{\ts_{i1}}{\ts_{01}}$. Therefore, $\hs_7$ is the
  transposition
  $(0,1)$ on $\Pi_\A$, and
  $$
\s_7^2(t)=\s_{49}(t) = \mtx{e^{\frac{17\pi i }{30}} & 0 \\ 0 & e^{\frac{-7 \pi i}{30}}} = \mtx{t_{11} &
0
\\ 0 & t_{00}}\,.
  $$
  However, the Galois symmetry does not hold for $\tt$ as
  $$
  \s^2_7(\tt) = \mtx{1 & 0 \\ 0 & e^{\frac{6\pi i}{5}}} \ne \mtx{\tt_{11} & 0 \\ 0 & \tt_{00}}\,.
  $$
  }
\end{example}
We close this section with the following proposition which provides a necessary and sufficient condition
for such Galois symmetry of the $T$-matrix $\tt$ of a modular category.
\begin{prop}\label{p:anomaly}
   Suppose $\A$ is a modular category over $\k$ with  Frobenius-Schur exponent $N$ and $T$-matrix
   $\tt=[
   \delta_{ij}\w_i]_{i,j \in \Pi_\A}$, and let $\zeta \in \k$ be a $6$-th root of the anomaly $\a=\frac{p^+_\A}{p^-_\A}$ of $\A$.
   Then for any $\s \in \AQab$ and $i \in \Pi_\A$,
   \begin{equation}\label{eq:ratio1}
     \frac{\w_{\hs(i)}}{\s^2(\w_i)} = \w_{\hs(0)} = \frac{\zeta}{\s^2(\zeta)}\,.
   \end{equation}
   Moreover, the following statements are equivalent:
   \begin{enumerate}
     \item[(i)] $\w_{\hs(0)}=1$ for all $\s \in \AQab$.
     \item[(ii)] $\s^2(\w_i) = \w_{\hs(i)}$ for all $\s \in \AQab$.
     \item[(iii)] $\left(\frac{p^+_\A}{p^-_\A}\right)^4=1$.
  \end{enumerate}
\end{prop}
\begin{proof} By \eqref{eq:repC},  the assignment
  $$
  \rho^{\zeta}(\fs)=s=\l\inv \ts, \quad  \rho^{\zeta}(\ft)=t=\zeta\inv \tt
   $$
   defines a modular representation of $\A$ where $\l = p_\A^+/\zeta^3$. For $\s \in \AQab$ and $i \in
   \Pi_\A$, Theorem II (iii) implies that
  $$
  \s^2\left(\frac{\w_i}{\zeta}\right) = \s^2(t_{ii}) = t_{\hs(i)\hs(i)} =  \frac{\w_{\hs(i)}}{\zeta}\,.
  $$
  Thus \eqref{eq:ratio1} follows as $\w_0 =1$.

  By \eqref{eq:ratio1}, the equivalence of (i) and (ii) is obvious. Statement (i) is equivalent to
  \begin{equation}\label{eq:zeta invariant1}
   \s^2(\zeta) =\zeta \quad\text{for all } \sigma \in \AQab\,.
  \end{equation}
  Since the anomaly $\a$ is a root of unity,  so is $\zeta$.  By Lemma
  \ref{app2}, \eqref{eq:zeta invariant1} holds if, and only if,  $\zeta^{24} = 1$  or  $\a^4 =1$.
\end{proof}
\begin{remark}
{\rm
  For a modular category $\A$ over $\BC$, it follows from \eqref{eq:charge} that the anomaly of $\A$ being
  a
  fourth root of unity is equivalent to  its  central charge $\bc$ being an integer modulo $8$.
  }
\end{remark}
\section{Galois symmetry of quantum doubles} \label{s:Quantum_double}
In this section, we provide a proof for Lemma \ref{l:double case} which is a special case of Theorem II
(iii)
and (iv), but which is also crucial to the proof of the theorem. We will invoke the machinery of \emph{generalized
Frobenius-Schur indicators} for spherical
fusion categories introduced in \cite{NS4}.

\subsection{Generalized Frobenius-Schur indicators} \label{ss:GFS}
Frobenius-Schur  indicators for group representations have been recently generalized to the
representations of Hopf algebras \cite{LM00}, and quasi-Hopf algebras \cite{MN05,  Sch04, NS2}. A
version
of  the 2nd Frobenius-Schur indicator was introduced in conformal field theory \cite{Bantay00}, and some
categorical
versions were studied in \cite{FGSV99, FucSch:CTCBC}. All these different contexts of
indicators are specializations of the Frobenius-Schur indicators for pivotal categories introduced in
\cite{NS1}.

The most recent introduction of the equivariant Frobenius-Schur indicators for semisimple Hopf algebras
by \cite{SZh} has motivated the discovery of generalized Frobenius-Schur indicators for pivotal
categories \cite{NS4}. The specialization of these generalized Frobenius-Schur indicators to spherical
fusion categories carries a natural action of $\SL{}$. This modular group action has played a crucial role for
the congruence subgroup theorem \cite[Thm. 6.8]{NS4} of the projective representation of $\SL{}$
associated with a modular category. These indicators also admit a natural action of $\AQab$ which will
be
employed to prove the Galois symmetry of quantum doubles in this section. For the
purpose
of this paper, we will only provide relevant details of generalized Frobenius-Schur indicators for our
proof  to be presented here. The readers are referred to \cite{NS4} for more
details.

Let $\CC$ be a strict spherical fusion category over $\k$  with Frobenius-Schur exponent $N$.
For any pair $(m,l)$ of integers, $V \in \CC$ and $\bX = (X, \s_X) \in Z(\CC)$, there is a naturally
defined $\k$-linear operator $E_{\bX, V}^{(m,l)}$ on the finite-dimensional $\k$-space $\CC(X, V^m)$
(cf.
\cite[Sect. 2]{NS4}). Here, $V^0=\1$, $V^m$ is the $m$-fold tensor product of $V$ if $m>0$, and $V^m =
(V\du)^{-m}$ if $m<0$.
The $(m,l)$-th \emph{generalized Frobenius-Schur indicator} for $\bX \in Z(\CC)$ and $V \in \CC$ is
defined
as
\begin{equation}\label{eq:E0}
\nu_{m,l}^\bX(V):=\Tr\left(E_{\bX, V}^{(m,l)}\right)
\end{equation}
where $\Tr$ denotes the ordinary trace map. In particular, for $m > 0$ and $f \in \CC(X, V^m)$, $E_{\bX,
V}^{(m,1)}(f)$ is the following composition:
$$
X \xrightarrow{X\o \db_{V\du}} X \o V\du \o V \xrightarrow{\s_X(V\du)\o V}  V\du \o X \o V \xrightarrow{
V\du\o f \o V} V\du \o V^m \o V \xrightarrow{\ev_V \o V^m} V^m \,.
$$
It can be shown by graphical calculus that for $m, l \in \BZ$ with $m \ne 0$,
\begin{equation}\label{eq:E1}
E_{\bX, V}^{(m,l)} = \left(E_{\bX, V}^{(m,1)}\right)^l \quad \text{and}\quad \left(E_{\bX,
V}^{(m,1)}\right)^{m N}=\id
\end{equation}
(cf. \cite[Lem. 2.5 and 2.7]{NS4}). Hence, for $m \ne 0$, we have
\begin{equation}\label{eq:E2}
\nu_{m,l}^\bX(V)=\Tr\left(\left(E_{\bX, V}^{(m,1)}\right)^l\right)\,.
\end{equation}
Note that $\nu_{m,1}^\1(V)$ coincides with the Frobenius-Schur indicator $\nu_m(V)$ of $V \in \CC$
introduced in \cite{NS1}.

\subsection{Galois group action on generalized Frobenius-Schur indicators}
Let $\KK(Z(\CC))$ denote the Grothendieck ring of $Z(\CC)$ and $\KK_\k(Z(\CC)) = \KK(Z(\CC)) \o_\BZ \k$.
For any matrix $y \in  \GLk{\Pi}$, we define the linear operator $F(y)$ on $\KK_\k(Z(\CC))$ by
$$
F(y)(j) = \sum_{i \in \Pi} y_{ij} i \quad\text{for all }j\in \Pi,
$$
where $\Pi = \Pi_{Z(\CC)}$.
Then $F: \GLk{\Pi} \to \Aut_\k(\KK_\k(Z(\CC))$ is a group isomorphism. In particular, every representation
$\rho:G \to \GLk{\Pi}$ of a group $G$ can be considered as a $G$-action on $\KK_\k(Z(\CC))$ through $F$. More
precisely, for $g \in G$, we define
$$
g j = F(\rho(g))(j) \quad\text{for all }j\in \Pi.
$$

Let $\ts$ and $\tt$ be the $S$ and $T$-matrices of $Z(\CC)$. The $\SL{}$-action on $\KK_\k(Z(\CC))$ associated
with the canonical modular representation $\rho_{Z(\CC)}$ of $Z(\CC)$ is then given by
\begin{equation}\label{eq:action1}
\fs j = \sum_{i \in \Pi} s_{ij} i \quad \text{and}\quad  \ft j = \w_j j\,,
\end{equation}
where $\tt=[\delta_{ij} \w_j]_{i,j \in \Pi}$  and $s = \frac{1}{\dim \CC} \,\ts$ (cf. \eqref{eq:can
normalization}). Note that $s \in \GLR{\Pi}{\BQ_N}$ by Theorem II (ii), since $N=\ord(\tt)$.

Now we extend the generalized indicator $\nu_{m,l}^\bX (V)$ linearly to a
functional $I_V((m,l), -)$ on $\KK_{\k}({Z(\CC)})$ via the  basis $\Pi$. For $V \in
\CC$,
$(m,l) \in
\BZ^2$ and $z=\sum_{i \in \Pi} \a_i i \in \KK_{\k}({Z(\CC)})$ for some $\a_i \in \k$,  we define
$$
I_V((m,l),z) = \sum_{i \in \Pi} \a_i\nu_{m,l}^{\bX_i}(V)
$$
where $\bX_i$ denotes an arbitrary object in the isomorphism class $i$. The $\SL{}$-actions on $\BZ^2$
and
on $\KK_\k({Z(\CC)})$ are related by these functionals on $\KK_\k({Z(\CC)})$. We summarize  some results
on
these generalized indicators relevant to the proof of Lemma \ref{l:double case} in the following theorem
(cf. Section 5  of \cite{NS4}):
\begin{thm} \label{t:1} Let $\CC$ be  a spherical fusion category $\CC$ over $\k$
with
Frobenius-Schur
exponent $N$. Suppose $z \in \KK_{\k}({Z(\CC)})$, $\bX=(X, \s_X) \in Z(\CC)$, $V \in \CC$, $(m,l) \in \BZ^2$ and
$J=\mtx{1&0\\0&-1}$.
Then:
  \begin{enumerate}
    \item[(i)] $\nu_{m,l}^\bX(V) \in \BQ_N$.
    \item[(ii)] $\nu_{1,0}^\bX (V) = \dim_\k \CC(X, V)$.
    \item[(iii)]  $I_V((m,l)\fg, z)= I_V((m,l), \fg^J  z)$ for $\fg \in \SL{}$,  where
        $\fg^J=J \fg J$.
  \end{enumerate}
  In particular,  $\AQab$ acts on the generalized Frobenius-Schur indicators $\nu_{m,l}^\bX(V)$. \qed
\end{thm}

For $\s \in \AQab$, $\gs =\s(s)s\inv$ is also given by
$$
(\gs)_{ij}=\e_\s(i) \delta_{\hs(i) j}
$$
for some sign function $\e_\s$ and permutation $\hs$ on $\Pi$ (cf. \eqref{eq:galois1},
\eqref{eq:galois2}
and \eqref{eq:galois3}). Define $\ff_\s = F(\gs)$.  Then
\begin{equation}\label{eq:action2}
\ff_\s  j = \e_\s(\hs\inv(j)) \hs\inv(j) \quad\text{for }j \in \Pi\,.
\end{equation}
Since the assignment $\AQab\to \GLR{\Pi}{\BZ},  \s \mapsto \gs$ is a representation of $\AQab$,
$$
\ff_\s \ff_\tau = \ff_{\s\tau} \quad\text{for all }\s, \tau \in \GalQ{N}.
$$
Therefore, by direct computation,
$$
\ff_{\s\inv} j = \ff_\s\inv j = \e_\s(j)\hs(j)\quad\text{for }j \in \Pi\,.
$$

\begin{remark} \label{r:equalf}
  Since $s \in \GLR{\Pi}{\BQ_N}$,  if  $\s, \s' \in \AQab$ such that $\s|_{\BQ_N} = \s'|_{\BQ_N}$, then
  $$
  G_\s = G_{\s'}\quad \text{and so}\quad \ff_\s = \ff_{\s'}\,.
  $$
\end{remark}

Now we can establish the following lemma which describes a relation between the $\AQab$-action on
$\KK_\k({Z(\CC)})$ and the $\SL{}$-action  in terms of the functionals $I_V((m,l),-)$.
\begin{lem} \label{l:2}
Let $V \in \CC$ and $a, l$ non-zero integers such that $a$ is relatively prime to $lN$. Suppose $\s \in
\AQab$ satisfies $\s|_{\BQ_N}=\s_a$. Then, for all  $z \in \KK_{\k}({Z(\CC)})$,
  $$
  I_V((a, l),  z)= I_V((1,0), \ft^{-al} \ff_\s z)\,.
  $$
\end{lem}
\begin{proof} Let $j \in \Pi$ and $\bX_j$ a representative of $j$.  By \eqref{eq:E1}, \eqref{eq:E2} and Theorem
\ref{t:1}(i), for any non-zero integer $m$, there is a linear operator $E_m
=E_{\bX_j, V}^{(m,1)}$ on a
finite-dimensional space such that $\left(E_m\right)^{mN}=\id$ and
$$
\nu_{m,k}^{\bX_j}(V)=\Tr(E_m^k) \in \BQ_N
$$
for all integers $k$. In particular, the eigenvalues of $E_m$ are $|mN|$-th roots of unity.

Suppose $\tau \in \AQab$ such that $\tau|_{\BQ_{|lN|}} = \sigma_a$. Then $\tau|_{\BQ_N} = \s_a =
\s|_{\BQ_N}$. Therefore,
\begin{equation}\label{eq:gal action1}
\s(\nu_{l,-1}^{\bX_j}(V)) = \tau(\Tr(E_l\inv))=\Tr(E_l^{-a})=\nu_{l,-a}^{\bX_j}(V) = I_V((l,-a), j)
\end{equation}
and
\begin{multline}\label{eq:gal action2}
\s(\nu_{1,l}^{\bX_j}(V)) = \s_a(\Tr(E_1^l))=\Tr(E_1^{la})=\nu_{1,la}^{\bX_j}(V)\\
 = I_V((1,la), j) = I_V((1,0) \ft^{la}, j) = I_V((1,0), \ft^{-la} j)\,.
\end{multline}
Here, the last equality follows from Theorem \ref{t:1}(iii).

On the other hand, by Theorem \ref{t:1}(iii), we have
$$
\nu_{1,l}^{\bX_j}(V)= I_V((1,l), j) = I_V((l,-1)\fs\inv, j)
=I_V((l,-1),\fs j)
=\sum_{i \in \Pi} s_{ij}\nu_{l,-1}^{\bX_i}(V).
$$
Therefore, \eqref{eq:gal action1} and Theorem \ref{t:1}(iii) imply
\begin{multline*}
\s(\nu_{1,l}^{\bX_j}(V))=\s \left( \sum_{i \in \Pi} s_{ij}\nu_{l,-1}^{\bX_i}(V)\right)
=\sum_{i \in \Pi} \e_{\s}(j) s_{i\hs(j)}\s(\nu_{l,-1}^{\bX_i}(V))\\
=\sum_{i\in \Pi} \e_{\s}(j) s_{i\hs(j)}I_V((l,-a), i)=
I_V((l, -a), \e_\s(j)\fs\, \hs(j)) \\
= I_V((l, -a), \fs (\ff_{\s\inv} j))
= I_V((l, -a)\fs\inv, \ff_{\s\inv} j) =
I_V((a, l), \ff_{\s\inv} j)\,.
\end{multline*}
It follows from \eqref{eq:gal action2} that for all $j \in \Pi$,
$$
I_V((a, l), \ff_{\s\inv} j)  = I_V((1,0), \ft^{-la} j)
$$
and so
$$
I_V((a, l), \ff_{\s\inv} z)  = I_V((1,0), \ft^{-la} z)
$$
for all $z \in \KK_\k(Z(\CC))$. The assertion follows by replacing $z$ with $\ff_\s z$.
\end{proof}
\begin{remark}
   {\rm Some related equalities for the representation categories of semisimple Hopf algebras were obtained in
   \cite[Cor 12.4]{SZh} with a similar strategy. Because of the conceptual differences of the definitions of
   generalized Frobenius-Schur indicators for  spherical fusion categories and the counterpart for semisimple
   Hopf algebras introduced in that paper, their approach generally cannot be adapted in fusion categories.}
\end{remark}
\subsection{Proof of Lemma \ref{l:double case}}
 Let $\s \in \AQab$ and $\s|_{\BQ_N} =  \s_a$ for
some integer $a$ relatively prime to $N$. Then $\s\inv|_{\BQ_N}=\s_b$ where $b$ is an inverse of $a$
modulo $N$. By Dirichlet's
theorem on primes in arithmetic progressions, there exists a prime $q$ such that $q \equiv b \mod N$ and $q \nmid a$.
By Lemma \ref{l:2}
and Theorem \ref{t:1}(iii), for $j \in \Pi$, we have
\begin{multline} \label{eq:nu11}
I_V((1,0), \ft\inv \ff_{\s} \ft^{q} \ff_{\s\inv} j)=
I_V((1,0), \ft^{-aq} \ff_{\s} \ft^{q} \ff_{\s\inv} j) \\
=I_V((a,q),  \ft^{q} \ff_{\s\inv} j)
=I_V((a,q)\ft^{-q},  \ff_{\s\inv} j)
=I_V((a,q-aq),  \ff_{\s\inv} j) \\
=  I_V((1,0), \ft^{-aq+a^2 q} \ff_\s \ff_{\s\inv} j)
=  I_V((1,0), \ft^{-1+a} j)
\,.
\end{multline}
Using \eqref{eq:action1} and \eqref{eq:action2}, we can compute directly the two sides of
\eqref{eq:nu11}.
 This implies
$$
\w_j\inv \w^{q}_{\hs(j)} \nu_{1,0}^{\bX_j} (V)= \w_j^{a-1}\nu_{1,0}^{\bX_j} (V)
$$
for all $V \in \CC$. Take $V=X_j$, the underlying $\CC$-object of $\bX_j$. We then have
$\nu_{1,0}^{\bX_j}(X_j)=\dim_\k \CC(X_j, X_j) \ge 1$.
Therefore,
we have $\w_j\inv \w^{q}_{\hs(j)} = \w_j^{a-1}$, and hence
$$
\w^{q}_{\hs(j)} = \w_j^{a}\quad\text{or}\quad \w_{\hs(j)}  =\w_j^{a^2}\,.
$$
This is equivalent to the equality
$$
\s^2(\tt) =\gs  \tt \gs\inv\,.
$$

Since $\tt s\tt s\tt=s$, we find
  \begin{multline}
    \gs s=\s(s) = \s(\tt s\tt s\tt)=\tt^a s\gs\inv \tt^a \gs s \tt^a \\ =
    \tt^a s\gs\inv \tt^{a^2b} \gs s \tt^a =\tt^a s (\gs\inv \tt^{a^2} \gs)^b s \tt^a = \tt^a s \tt^b s
    \tt^a\,.
  \end{multline}
  Therefore,
  $$
  \gs=\tt^a s \tt^b s \tt^a s\inv\,.
  $$
  This completes the proof of Lemma \ref{l:double case}. \qed

\section{Anomaly of modular categories }\label{s:anomaly}
In this section, we apply the congruence property and Galois symmetry of a modular category (Theorem II)
to
deduce some arithmetic relations among the global dimension, the Frobenius-Schur exponent and the order
of the  anomaly.

Let $\A$ be a modular category over $\k$ with Frobenius-Schur exponent $N$. Since $d(V) \in \BQ_N$ for $V \in
\A$ (cf. \cite[Prop. 5.7]{NS4}), the anomaly $\a=\frac{p_\A^+}{p_\A^-}$ of $\A$ is a root of unity in
$\BQ_N$. Therefore, $\a^{N}=1$ if $N$ is even, and
$\a^{2N}=1$ if $N$ is odd.

Let us define $J_\A = (-1)^{1+\ord \a}$ to record the parity of the order of the anomaly $\a$ of $\A$. Note
that $J_\A$ is intrinsically defined by $\A$.
It
will become clear that $J_\A$ is closely related to the Jacobi symbol $\Jac{*}{*}$ in number theory.
When
$4 \nmid N$,  $J_\A$ determines whether $\dim \A$ has a square root in $\BQ_N$.

\begin{thm}\label{t:Jacobi}
  Let $\A$ be a modular category over $\k$ with  Frobenius-Schur exponent $N$ such that $4 \nmid N$.
  Then
  $J_\A \dim \A$ has a square root in $\BQ_N$ and $-J_\A \dim \A$ does not have any square root
  in
  $\BQ_N$.
\end{thm}
\begin{proof} Let $\zeta \in \k$ be a 6-th root of the anomaly $\a=\frac{p_\A^+}{p_\A^-}$ of $\A$. By Corollary \ref{c1},
there
exists a 12-th root of unity $x \in \k$ such that
$$
\left(\frac{x}{\zeta}\right)^N=1 \quad\text{and}\quad \frac{x^3 p_\A^+}{\zeta^3}\in  \BQ_N\,.
$$
Note that $\left(\frac{p_\A^+}{\zeta^3}\right)^2 = \dim \A$.

   Set $N'=N$ if $N$ is odd and $N'=N/2$ if $N$ is even. In particular, $N'$ is odd. Then
   $(\frac{x}{\zeta})^{N'} =\pm 1$ and so
   $$
   \a^{N'} = \zeta^{6N'} = x^{6N'}=x^6\,.
   $$
    By straightforward verification, one can show that $x^6=J_\A$. Therefore,
    $$
    \left(\frac{x^3 p_\A^+}{\zeta^3}\right)^2=x^6 \dim \A  =J_\A \dim \A\,.
    $$

   Suppose $-J_\A \dim \A$ also has a square root in $\BQ_N$. Since $J_\A \dim \A$ has a square root in
   $\BQ_N$,  so does $-1$. Therefore, $4 \mid N$, a contradiction.
\end{proof}

When $\dim \A$ is an odd integer, we will show that $J_\A = \Jac{-1}{\dim \A}$. Let us fix our
convention
in the following definition for the remainder of this paper.

\begin{defn}\label{d:interal_cats}
{\rm
  Let $\A$ be a modular category over $\k$.
 \begin{enumerate}
   \item[(i)] $\A$ is called \emph{mock integral} if its global dimension $\dim \A$ is an integer.
   \item[(ii)] $\A$ is called  \emph{integral} if $d(V) \in \BZ$ for all $V \in \A$.
 \end{enumerate}
 }
 \end{defn}
\begin{remark}
{\rm The standard definition of integral fusion categories is defined in terms of Frobenius-Perron dimensions. Following \cite{ENO}, a fusion category $\CC$ is called integral (resp. weakly integral) if $\FPdim V \in \BZ$ for all $V \in \CC$ (resp. $\FPdim \CC \in \BZ$). Moreover, any weakly integral spherical fusion category $\CC$ satisfies the pseudo-unitary condition: $\FPdim \CC = \dim \CC$. Therefore, weakly integral modular categories are obviously mock integral. The Deligne product of the Fibonacci modular category (cf. \cite[5.3.2]{RSW}) with its Galois conjugate is a mock integral modular category but not weakly integral.
}
 \end{remark}
It follows from \cite[Lem. A.1]{HR} and \cite[Prop. 8.24]{ENO} that $d(V) \in \BZ$ for all objects $V$ in a modular category $\A$ if,
and only if, $\FPdim V \in \BZ$ for all $V \in \A$ . Therefore, these two definitions of \emph{integral} modular categories are equivalent. A weakly integral modular category can also be characterized by the integrality of $d(V)^2$ as in the following lemma.
\begin{lem}
  A modular category $\A$ over $\k$ is weakly integral if, and only if,  $d(V)^2 \in \BZ$ for any simple object $V \in \A$.
\end{lem}
\begin{proof}
 By the modularity of $\A$, $\FPdim \A = \dim \A/d(U)^2$ for some simple object $U$. If $d(V)^2 \in \BZ$ for all simple objects $V \in \A$, then $\dim \A \in \BZ$ and hence $\FPdim \A \in \BZ$. Conversely, if $\FPdim \A \in \BZ$, then $\FPdim \A =\dim \A$ and $(\FPdim V)^2 \in \BZ$ for all simple objects $V \in \A$ by  \cite[Prop. 8.24 and 8.27]{ENO}. Since $d(V)^2 \le (\FPdim V)^2$, the pseudo-unitarity of $\A$ implies $d(V)^2 = (\FPdim V)^2 \in \BZ$.
\end{proof}

\begin{prop} \label{p:1st Jacobi}
  Let $\A$ be a mock integral modular category over $\k$ with Frobenius-Schur exponent $N$ and odd
  global
  dimension $\dim \A$. Then $J_\A = \Jac{-1}{\dim \A}$. In particular,
  $$
  J_\A =\left\{\begin{array}{ll}
    1 & \text{if }\dim \A \equiv 1 \mod 4, \\
    -1 & \text{if }\dim \A \equiv 3 \mod 4.
  \end{array}\right.
  $$
  Moreover, the square-free part of $\dim \A$ is a divisor of $N$.
\end{prop}
\begin{proof} We may simply assume $\A$ contains a non-unit simple object. By \cite[Thm. 5.1]{Etingof02},
$N$ divides $(\dim \A)^3$.  In particular, $N$ is odd. It follows from the proof of \cite[Prop. 2.9]{ENO}
that for any embedding $\varphi:\BQ_N \to \BC$, $\varphi(d_i)$ is real for $i \in \Pi_\A$, and so $\varphi(\dim \A) > 1$. We can identify $\BQ_N$ with $\varphi(\BQ_N)$.

  If $\dim \A$ is the square of an integer, then $J_\A=1$ by Theorem \ref{t:Jacobi}, and $\Jac{-1}{\dim
  \A}=1$. In this case, the last statement is trivial.
  Suppose $\dim \A$ is not the square of any integer. It follows from
  Theorem \ref{t:Jacobi} that $\BQ(\sqrt{J_\A \dim \A})$ is a quadratic subfield of $\BQ_N$.  Note that
  $\BQ(\sqrt{p^*})$ is the unique quadratic subfield of $\BQ_{p^\ell}$ for any odd prime $p$ and positive
  integer $\ell$ (cf. \cite{Wash}),
  where $p^* =
  \Jac{-1}{p} p$, and that $\BQ(\sqrt{m}) \ne \BQ(\sqrt{m'})$ for any two distinct square-free integers
  $m,m'$. Let $p_1, \dots, p_k$ be the distinct prime factors of $N$.
  By counting the order 2 elements of $\GalQ{N}$, the quadratic subfields of $\BQ_N$ are of the form
  $\BQ(\sqrt{d^*})$ where $d$ is a positive divisor of $p_1\cdots p_k$, and $d^* = \Jac{-1}{d} d$.

  Let $a$ be the square-free part of $\dim \A$. Then $\Jac{-1}{\dim \A} = \Jac{-1}{a}$ and
  $\BQ(\sqrt{J_\A a})  = \BQ(\sqrt{J_\A \dim \A})$. By the preceding paragraph, $a \mid p_1\cdots p_k$
  and
  $J_\A = \Jac{-1}{a}$.
\end{proof}
\begin{remark}\label{r:6.5}
 {\rm In \cite{SZ2}, integral modular categories with the special Galois property ($\dagger$) $\s(\tilde
 s_{ij})=\tilde s_{\hs(i)j}$ were discussed. These conditions are not satisfied by some common modular
 categories such as the Ising and Fibonacci modular categories. However,  for semisimple quasi-Hopf algebras
 with modular module categories, the
first statement of the preceding proposition was proved  in \cite[Thm. 5.3]{SZ2}.

After our preprint was posted on the arXiv in early 2012, a number of results, which were not in \cite{SZ2},
appear in the serious revision \cite{SZ3} of \cite{SZ2} published in 2013. In \cite[Thm 2.6 and Prop 3.5]{SZ2,
SZ3}),  the same statement was established for integral modular categories satisfying ($\dagger$) by considering the quadratic subfields of $\BQ_N$ but using a different approach.}
\end{remark}
The following proposition on modular categories is a slight variation of \cite[Prop. 3]{CG99}, and it was
essentially proved [loc. cit.] under the assumption of Galois symmetry, which has been proved in the
previous
sections.
\begin{prop}\label{p:anomaly1}
  Let $\A$ be a modular category over $\k$, and $\rho$ a modular representation of $\A$. Set $s=\rho(\fs)$, $t=
  [\delta_{ij}t_i]_{i,j \in \Pi_\A}= \rho(\ft)$,  $n=\ord(t)$ and
  $$
  \BK_b=\BQ\left(\frac{s_{ib}}{s_{0b}}\big| i \in \Pi_\A\right) \quad \text{for }\, b \in \Pi_\A.
  $$
  \begin{enumerate}
    \item[(i)]  Then, for $\s \in \Gal(\BQ_n/\BK_b)$, $\s^2(t_b) = t_b$.
     \item[(ii)] If $\A$ is  integral, then the anomaly $\a=\frac{p^+_\A}{p^-_\A}$ of $\A$ is a $4$-th root
         of unity.
         \item[(iii)] Let $\BK = \BQ\left(\frac{s_{ib}}{s_{0b}}\big| i,b \in \Pi_\A\right)$, and $k$ the
             conductor of $\BK$, i.e.  the smallest positive integer $k$ such that $\BK \subseteq \BQ_k$.
     Then,  $\Gal(\BQ_n/\BK)$ is an elementary 2-group, and $|\Gal(\BQ_n/\BQ_k)|$ is a divisor of $8$.
     Moreover,  $\frac{n}{k}$ is a divisor of $24$, and $\gcd\left(\frac{n}{k}, k\right)$ divides $2$.
  \end{enumerate}
\end{prop}
\begin{proof}
(i) Let $\s \in \Gal(\BQ_{n}/\BK_b)$ and $\e_\s$ the sign function determined by $s$ (cf. \ref{eq:galois2}).
Suppose $s^2 = \sgns C$ where $\sgns=\pm 1$. Then, by
  \eqref{eq:galois1},
$$
\frac{\sgns}{s_{0b}^2} = \sum_{i \in \Pi_\A}\frac{s_{ib} s_{ib^*}}{s_{0b}^2}=\sum_{i \in
\Pi_\A}\left(\frac{s_{ib}}{s_{0b}}\right)\left(\frac{s_{ib^*}}{s_{0b}}\right)=\sum_{i \in
\Pi_\A}\left(\frac{s_{ib}}{s_{0b}}\right)\left(\frac{s_{i^*b}}{s_{0b}}\right) \in \BK_b\,.
$$
Therefore, $s^2_{0b} \in \BK_b$ and so $\s(s^2_{0b})=s^2_{0b}$. Since
$\s(s_{0b})=\e_\s(b)s_{0\hs(b)}$, $s_{0\hs(b)} = \e s_{0b}$ for some sign $\e$. Now, for $i \in \Pi_\A$,
$$
\frac{s_{ib}}{s_{0b}} = \s\left(\frac{s_{ib}}{s_{0b}}\right)= \frac{s_{i\hs(b)}}{s_{0\hs(b)}} =
  \frac{\e s_{i\hs(b)}}{s_{0b}}\,.
  $$
  Thus, $s_{ib}= \e s_{i \hs(b)}$ for all $i \in \Pi_\A$. If $\hs(b) \ne b$, then the $b$-th and the
  $\hs(b)$-th
  columns of $s$ are linearly dependent but this contradicts the invertibility of $s$. Therefore, $\hs(b) =b$
  and hence, by Theorem II (iii), $\s^2(t_b) =t_{\hs(b)} = t_b$. \smallskip\\
  (ii) If $\A$ is integral, then $\BK_0 = \BQ$ and hence $\s^2(t_0) = t_0$ for all $\s \in \Gal(\BQ_n/\BQ)$.
  Recall from Section \ref{ss:MTC} that $t_0 =  x /\zeta$ for some $6$-th root $\zeta$ of $\a$ and some $12$-th
  root of unity $x \in \k$. By Lemma \ref{app2}, $x /\zeta$ is a $24$-th root of unity. Therefore,
  $$
  \a^4 = \zeta^{24} = \left(\zeta/x\right)^{24} =1\,.
  $$
  (iii) By (i), for $\s \in \Gal(\BQ_n/\BK)$, $\s^2(t_b)=t_b$ for all $b \in \Pi_\A$. Since $\BQ_n$ is
  generated by $t_b$ ($b \in \Pi_\A$), $\s^2 = \id$. Therefore, $\Gal(\BQ_n/\BK)$ is an elementary $2$-group,
  and so is $\Gal(\BQ_n/\BQ_k)$. Thus, for any integer $a$ relatively prime to $n$ such that $a \equiv 1 \mod
  k$, $a^2 \equiv 1 \mod n$. By Lemma \ref{l:app3},  we have $n/k$ is a divisor of $24$ and $\gcd(n/k, k) \mid
  2$. Moreover,
  $|\Gal(\BQ_n/\BQ_k)|=\phi(n)/\phi(k)$ is a divisor of $8$.
\end{proof}
\begin{remark}
  {\rm The proof of the preceding proposition is a mere adaptation of  \cite[Prop. 3]{CG99}. For integral modular categories satisfying ($\dagger$) (see Remark \ref{r:6.5}), Proposition \ref{p:anomaly1} (ii) and (iii) also appear in the final version of
  \cite[Thms. 2.3.2, 3.4]{SZ3} with similar ideas. The following corollary was also established for factorizable quasi-Hopf algebras in \cite[Thm. 4.3]{SZ2, SZ3} with a different approach.}
\end{remark}
\begin{cor} \label{c:anomaly1}
  Let $\A$ be an integral modular category with anomaly $\a=\frac{p_\A^+}{p_\A^-}$. If $\dim \A$ is odd, then
 $\a = \Jac{-1}{\dim \A}$.
\end{cor}
\begin{proof}  If $\dim \A$ is odd, then so is the Frobenius-Schur exponent $N$ of $\A$ as $N \mid (\dim
\A)^3$.
  Since $\a \in \BQ_N$ and $\a^4 = 1$, $\a^2=1$. It follows from Proposition \ref{p:1st Jacobi} that
  $$
  \a = (-1)^{1 +\ord \a}= J_\A = \Jac{-1}{\dim\A}\,. \qedhere
  $$
\end{proof}

The Ising modular category is an example of weakly integral modular category (cf.
\cite[5.3.4]{RSW})
and its central charge is $\bc=\frac{1}{2}$. Therefore,  its anomaly is $e^{\pi i/4}$, an  eighth root
of
unity, and this holds for every weakly integral modular category.
\begin{thm}\label{t:weakly integral}
  The anomaly $\a = p_\A^+/p_\A^-$ of any weakly integral modular category $\A$ is an eighth root of unity.
\end{thm}
\begin{proof}
  Suppose $\zeta \in \k$ is a 6-th root of the anomaly $\a$ of a weakly integral modular category $\A$.
  Then
  $\l = p_\A^+/\zeta^3$ is a square root of $\dim \A$. Consider the modular representation $\rho^\zeta$
  of
  $\A$ given by
  $$
  \rho^\zeta : \fs \mapsto s:=\frac{1}{\l} \ts, \quad \ft \mapsto t:=\frac{1}{\zeta}\tt\,.
  $$ Let $\tt=[\delta_{ij} \w_i]_{i,j \in \Pi_\A}$ be the $T$-matrix of $\A$. Since
  $s^2_{0i}=\frac{d^2_i}{\dim
  \A} \in \BQ$, for $\s \in \AQab$,
  $$
  s^2_{0i} = \s(s^2_{0i}) = s^2_{0\hs(i)}
  $$
  or $d^2_i =  d^2_{\hs(i)}$ for all $i \in \Pi_\A$.
  By Theorem II (iii),
  $$
   \s^2\left(\sum_{i \in \Pi_\A} d_i^2 \frac{\w_i}{\zeta} \right)=
  \sum_{i \in \Pi_\A} d_i^2 \frac{\w_{\hs(i)}}{\zeta}
  =  \sum_{i \in \Pi_\A} d_{\hs(i)}^2 \frac{\w_{\hs(i)}}{\zeta}
  = \sum_{i \in \Pi_\A} d_i^2 \frac{\w_i}{\zeta}
  \,.
  $$
  Thus, we have
  $$
  \frac{\s^2(p_\A^+)}{p_\A^+} = \frac{\s^2(\zeta)}{\zeta}\,.
  $$
  Since $\dim \A$ is a positive integer, $\s^2(\l) = \l$ and so
  $$
  \frac{\s^2(\zeta^3)}{\zeta^3} = \frac{\s^2(p_\A^+/\l)}{p_\A^+/\l} = \frac{\s^2(p_\A^+)}{p_\A^+} =
  \frac{\s^2(\zeta)}{\zeta}\,.
  $$
  Therefore, we find $\frac{\s^2(\zeta^2)}{\zeta^2} = 1$ for all $\s \in \AQab$. It follows from Lemma
  \ref{app2} that $\zeta^{48}=1$ and so $\a^8=1$.
\end{proof}

Corollary \ref{c:anomaly1} and the Cauchy theorem for Hopf algebras \cite{KSZ} as
well as quasi-Hopf algebras \cite{NS3} suggest a more general version of Cauchy theorem may hold for
spherical fusion categories or modular categories over $\k$. We finish this paper with two equivalent
questions.
\begin{q} \label{q:1}
  Let $\CC$ be a spherical fusion category over $\k$ with Frobenius-Schur exponent $N$. Let $\OO$ denote
  the ring of integers of $\BQ_N$. Must the principal ideals $\OO (\dim \CC)$ and $\OO N$ of $\OO$ have
  the
  same prime ideal factors?
\end{q}
Since $Z(\CC)$ is a modular category over $\k$ and $(\dim \CC)^2 = \dim Z(\CC)$, the preceding question
is
equivalent to
\begin{q}
  Let $\A$ be a modular category over $\k$ with Frobenius-Schur exponent $N$. Let $\OO$ denote the ring
  of
  integers of $\BQ_N$. Must the principal ideals $\OO (\dim \A)$ and $\OO N$ of $\OO$ have the same prime
  ideal factors?
\end{q}
By \cite{Etingof02}, $\frac{(\dim \A)^3}{N} \in \OO$. Therefore, the prime ideal factors of $\OO N$ are a
subset
of $\OO \dim \A$. The converse is only known be true for the representation categories of semisimple
quasi-Hopf
algebras by \cite[Thm. 8.4]{NS3}. Question \ref{q:1} was originally raised in \cite[Qu. 5.1]{EG99} for
semisimple Hopf algebras which had been solved in \cite[Thm. 3.4]{KSZ}.

\section*{Appendix}
The first lemma in this Appendix could be known to some experts. An analogous result for $\PSL$ was proved by
Wohlfahrt
\cite[Thm. 2]{Woh} (see also Newman's proof \cite[Thm. VIII.8]{New}). However, we do not see the
lemma as an immediate consequence of Wohlfahrt's theorem for $\PSL$.

\begin{Applem}\label{l:app1}
 Let $H$ be a  congruence normal subgroup of $\SL{}$. Then the level of $H$ is equal to the order of
 $\ft
 H$ in $\SL{}/H$.
\end{Applem}
\begin{proof}
   Let $m$ be the level of $H$ and $n=\ord \ft H$. Since $\ft^m \in \G(m) \le H$,  $\ft^m \in H$  and
   hence
   $n \mid m$.

  Suppose $\fg =\mtx{a & b \\ c & d} \in \G(n)$. Since $ad-bc=1$, by Dirichlet's theorem, there exists a
  prime $p \nmid m$ such that $p=d+kc$ for some integer $k$. Then,
  $$
  \ft^{-k}\fg \ft^k = \mtx{a' & b'  \\ c & p} \in \G(n)
  $$
  for some integers $a', b'$. In particular,
  $$
  a'p-b'c=1,\quad  p \equiv a' \equiv 1 \mod n \quad\text{and} \quad c \equiv b' \equiv 0 \mod n.
   $$
   Since $p \nmid m$, there exists an integer $q$ such that $p q \equiv 1 \mod m$. Thus, $pq \equiv 1
   \mod
   n$ and so $q \equiv 1 \mod n$. One can verify directly that
  $$
  \mtx{a' & b' \\ c & p} \equiv \ft^{b'q}\fs\inv \ft^{(-c+1)p}\fs\ft^{q}\fs \ft^p \mod m\,.
  $$
  Therefore,
  $$
      \ft^{-k} \fg \ft^k H = \ft^{b'q} \fs\inv \ft^{(-c+1)p}\fs\ft^{q}\fs \ft^p H
      =  \fs\inv  \ft  \fs  \ft  \fs  \ft H =  \fs\inv   \fs H =H\,.
  $$
 This implies $\ft^{-k} \fg \ft^k \in H$, and hence $\fg \in H$. Therefore, $\G(n) \le H$ and so $m \mid
 n$.
\end{proof}
The following fact should be well-known. We include the proof here for the convenience of the reader.
\begin{Applem}\label{app2}
Let $\zeta$ be a root of unity in $\k$.  Then $\s^2(\zeta)=\zeta$ for all $\s \in \AQab$ if, and only
if,
$\zeta^{24}=1$.
\end{Applem}
\begin{proof}
  Let $m$ be the order of $\zeta$. Then $\Gal(\BQ(\zeta)/\BQ) \cong U(\BZ_m)$.
  Note that the group $U(\BZ_m)$  has exponent $\le 2$ if and only if $m \mid 24$. Since $\BQ(\zeta)$ is
  a
  Galois extension over $\BQ$, the restriction map $\AQab \xrightarrow{res} \Gal(\BQ(\zeta)/\BQ)$ is
  surjective. Thus, if $\s^2(\zeta) = \zeta$ for all $\s \in \AQab$, then the exponent of
  $\Gal(\BQ(\zeta)/\BQ)$ is at most 2, and hence  $m \mid 24$. Conversely, if $m \mid 24$, then the
  exponent of $\Gal(\BQ(\zeta)/\BQ)$ is at most 2, and so $\s^2(\zeta) = \zeta$ for all $\s \in \AQab$.
\end{proof}
The next lemma is a variation of the argument used in the proof of \cite[Prop. 3]{CG99}.
\begin{Applem}\label{l:app3}
Let $k$ be a positive divisor of a positive integer $n$. Suppose that for any integer $a$ relatively prime to $n$
such that $a \equiv 1 \mod k$, $a^2 \equiv 1 \mod n$. Then $\gcd(n/k, k) \mid 2$  and $n/k$ is a divisor of
$24$. Moreover, $\phi(n)/\phi(k)$ is a divisor of $8$.
\end{Applem}
 \begin{proof}
Let $\pi : U(\BZ_n) \to U(\BZ_k)$ be the reduction map. The assumption  implies that $\ker \pi$ is an
elementary 2-group. It follows from the exact sequence
$$
0 \to \ker \pi \to U(\BZ_n) \xrightarrow{\pi} U(\BZ_k) \to 0
$$
that $\phi(n)/\phi(k)$ is a power of 2, and so is $\gcd(n/k, k)$. Thus, if $2 \nmid \gcd(n/k, k)$, then
$\gcd(n/k, k)=1$. By the Chinese remainder theorem,
for any integer $y$ relatively prime to $n/k$, there exists an integer $a$ such that $a \equiv y \mod n/k$, and
$a \equiv 1 \mod k$. Thus, $a^2 \equiv 1 \mod n$, and hence $y^2 \equiv 1 \mod n/k$. This implies the exponent
of  $U(\BZ_{n/k})$ is at most 2, and so  $\frac{n}{k}\mid 24$. Moreover, $\frac{\phi(n)}{\phi(k)} = \phi(n/k)$
is a factor of $8$.

Suppose $2 \mid \gcd(n/k, k)$. Then $k = 2^u k\rq{}$ for some positive integer $u$ and odd integer $k\rq{}$.
The aforementioned conclusion implies $n = 2^v n\rq{} k\rq{}$ where $v > u$ and $\gcd(n\rq{}, 2^v k\rq{})=1$.
By the Chinese remainder theorem, the given condition implies the  kernel of the reduction map $ U(\BZ_{2^v})
\to U(\BZ_{2^u})$ is an elementary 2-group. Therefore,  $2 \le v \le 3$ if $u=1$, and $v=u+1$ if $u > 1$. In
both cases, $\gcd(n/k, k) = 2$ and $\frac{\phi(2^v)}{\phi(2^u)}$ is a divisor of $4$. By the aforementioned
argument, for any integer $y$ relatively prime to  $n\rq{}$, $y^2 \equiv 1 \mod n\rq{}$. Therefore, $n\rq{}
\mid 24$ and hence $n\rq{} \mid 3$. Thus, $n/k = n\rq{} 2^{v-u} \mid 12$, and
$$
\frac{\phi(n)}{\phi(k)}=\phi(n')\,\frac{\phi(2^v)}{\phi(2^u)}
$$
is also a divisor of $8$.
\end{proof}
\vspace{0.3cm}
\begin{Ak} \hfill{\mbox{}}

{\rm
  Part of this paper was carried out while the third author was visiting the National Center for Theoretical
  Sciences and Shanghai University. He would like to thank these institutes for their generous
  hospitality,
  especially  Ching-Hung Lam, Wen-Ching Li and Xiuyun Guo for being wonderful hosts. He particularly
  thanks Ling Long for many invaluable discussions,  and Eric Rowell for his suggestions and stimulative
  discussions after the first version of this paper was posted on the arXiv.
  }
\end{Ak}

\providecommand{\bysame}{\leavevmode\hbox to3em{\hrulefill}\thinspace}
\providecommand{\href}[2]{#2}

\end{document}